\documentclass[11pt,reqno]{amsart}

\usepackage{enumerate} 
\usepackage{graphicx} 
\usepackage{xcolor}
\usepackage{amsthm}
\usepackage{amsmath,amsfonts,amssymb,mathrsfs}
\usepackage{mathtools}
\usepackage{enumitem}
\usepackage{ifthen}
\usepackage{hyperref}
\usepackage[margin=1.5in]{geometry}
\usepackage{caption}
\usepackage{sidecap}
\usepackage{rotating}
\usepackage{color}
\usepackage{cleveref}
\usepackage{esint}
\usepackage[T1]{fontenc}
\usepackage{lmodern}

\provideboolean{shownotes} 
\setboolean{shownotes}{true} 

\newcommand{\margnote}[1]{
\ifthenelse{\boolean{shownotes}}
{\marginpar{\raggedright\tiny\texttt{#1}}}
{}
}
\newcommand{\hole}[1]{
\ifthenelse{\boolean{shownotes}}
{\begin{center} \fbox{ \rule {.25cm}{0cm} \rule[-.1cm]{0cm}{.4cm}
\parbox{.85\textwidth}{\begin{center} \texttt{#1}\end{center}} \rule
{.25cm}{0cm}}\end{center}} {} }

\numberwithin{equation}{section}
\newtheorem{theorem}{Theorem}[section]
\newtheorem{lemma}{Lemma}[section]
\newtheorem{corollary}{Corollary}[section]
\newtheorem{proposition}{Proposition}[section]
\newtheorem{remark}{Remark}[section]
\newtheorem{definition}{Definition}[section]

\newcommand{\R}{\mathbb R}

\newcommand{\bq}{\begin{equation}}
\newcommand{\eq}{\end{equation}}

\newcommand{\nn}{\notag}

\newcommand{\Kg}{\mathbf{K}_g}
\newcommand{\J}{\mathbf{\mathcal{J}}}
\newcommand{\Jd}{\mathbf{\mathcal{J}}_\delta}

\DeclareMathOperator\supp{supp}

\DeclareMathOperator*{\osc}{osc}

\title[The two-phase Alt--Phillips problem for quasilinear operators]{The two-phase Alt--Phillips problem\\ for quasilinear operators}

\author[Y. Alamri]{Yousef Alamri}
\address{Applied Mathematics and Computational Sciences (AMCS), Computer, Electrical and Mathematical Sciences and Engineering Division (CEMSE), King Abdullah University of Science and Technology (KAUST), Thuwal, 23955-6900, Kingdom of Saudi Arabia}{} 
\email{yousef.alamri@kaust.edu.sa} 

\author[J.M. Urbano]{Jos\'{e} Miguel Urbano}
\address{Applied Mathematics and Computational Sciences (AMCS), Computer, Electrical and Mathematical Sciences and Engineering Division (CEMSE), King Abdullah University of Science and Technology (KAUST), Thuwal, 23955-6900, Kingdom of Saudi Arabia; and CMUC, Department of Mathematics, University of Coimbra, 3000-143 Coimbra, Portugal}
\email{miguel.urbano@kaust.edu.sa}

\begin{document}

\subjclass[2020]{Primary 35R35. Secondary 35A21, 35J70}




\keywords{Alt--Phillips functional; Free boundary problems; Hausdorff estimates; Degenerate and singular PDEs}

\begin{abstract}
We establish interior regularity and growth estimates for the sign-changing minimizers of the $p-$singular or $p-$degenerate quasilinear Alt--Phillips functional throughout the full range of $1<p<\infty$ and of the nonlinearity power $0<\gamma<p$. As a consequence, we show that the nonbranching hypersurface is locally $C^{1,\alpha}$ rectifiable. In addition, we obtain local density estimates, from which we deduce the local finiteness of the $\big(N-1+\gamma(p-1)/(p-\gamma)\big)-$dimensional Hausdorff measure of the remaining free boundaries for a restricted range of $\gamma$.
\end{abstract}

\date{\today}

\maketitle

\section{Introduction} \label{Intro}
Let $\Omega$ be an open and bounded subset of $\R^N$ with a Lipschitz boundary. We study local minimizers of the functional
\begin{equation} \label{J}
    \J[u]:= \int_\Omega  \dfrac{|\nabla u|^p}{p} + F_\gamma(u) 
  dx,  \quad \quad 1 < p <\infty ,
\end{equation}
with the continuous potential 
\begin{align} \label{Fg}
    F_\gamma(u) &:=  \lambda_+ (u_+)^\gamma + \lambda_- (u_-)^\gamma, \quad \  0< \gamma < p,
\end{align}
where $\lambda_+,\lambda_->0$ and the usual convention $u_\pm := \max\{\pm u,0\}$ is used. The minimization is over the admissible set of functions
\begin{equation} \label{Kg}
    \Kg(\Omega) := \left\{ u \in W^{1,p}(\Omega): u-g \in W_0^{1,p}(\Omega)\right\},
\end{equation}
for a given $g \in W^{1,p}(\Omega)$. The first term in \eqref{J} is minimized in $\Omega$ by a $p$-harmonic function satisfying $u = g$ on $\partial \Omega$. The value of such a minimizer in $\Omega$ is influenced by the sign of $g$, as well as the characteristics of the potential term $F_\gamma(u)$. If no sign condition is imposed on $g$, a minimizer may develop positive, negative, and zero phases in $\Omega$, namely:
\[
\{u>0\}:=\{x\in\Omega:\,u(x)>0\},\qquad
\{u<0\}:=\{x\in\Omega:\,u(x)<0\},
\]
and
\[
\{u=0\}:=\{x\in\Omega:\,u(x)=0\},
\]
which give rise to the \emph{free boundaries}
\begin{equation}\label{All-FB}
    \Gamma:=\big(\partial\{u>0\}\cup\partial\{u<0\}\big)\cap\Omega,
\end{equation}
that are located at the interface among these phases and are not known \textit{a priori}. Pending continuity of minimizers, the noncoincidence sets $\{u>0\}$ and $\{u<0\}$ are open relative to $\Omega$ and thus \eqref{All-FB} are the boundaries of $N$-dimensional sets. 

The value of $\gamma$ in \eqref{Fg} plays a crucial role in determining the properties of minimizers and the associated free boundaries. For one, the functional \eqref{J} loses convexity once $\gamma < 1$, and hence there is no guarantee for the uniqueness of minimizers in this range. Additionally, if one examines the Euler-Lagrange equation verified by any minimizer of $\J[\cdot]$, that is,
\begin{equation} \label{EL-1}
    \Delta_p u = F_\gamma'(u) \quad \text{in } \mathcal{D}'\big((\{u>0\}\cup\{u<0\})\cap\Omega\big).
\end{equation}
where 
\begin{equation} \label{RHS-F}
    F_\gamma'(u) = \gamma \left[\lambda_+ u^{\gamma-1} \chi_{\{u>0\}} - \lambda_- (-u)^{\gamma-1}\chi_{\{u<0\}}\right],
\end{equation}
one immediately notices the singularity in the equation along the free boundaries if $\gamma <1$, regardless of the sign of $u$. Yet, $F^{\prime}_\gamma(u)$ is bounded for all $\gamma \geq 1$ provided $u$ is so. In \eqref{EL-1}, $\Delta_pu := \text{div}(|\nabla u|^{p-2} \nabla u)$ is the well-known $p$-Laplacian, a singular operator in the range $1 < p < 2$, linear if $p=2$, and degenerate in the range $2<p<\infty$. 

If $u$ is a minimizer of \eqref{J}, the free boundaries \eqref{All-FB} can be decomposed (formally at the moment) into two complementary sets based on whether $\nabla u = 0$ or $\nabla u \neq 0$. Throughout, we shall denote the former portion by  
\begin{equation} \label{FB0}
    \Gamma_0 := \Gamma \cap \{ \nabla u = 0 \}. 
\end{equation}

Observe that the complement of $\Gamma_0$ in $\Gamma$ is empty if the minimizer of $\J[\cdot]$ has a sign (in which case the problem is posed in the one-phase setting, where automatically $u = \nabla u = 0$ on the free boundary) or the point of consideration belongs to the \emph{one-phase} subset of \eqref{All-FB} 
\[ \Big[\big( \partial \{u>0\} \backslash \partial \{u<0\} \big) \cup \big( \partial \{u<0\} \backslash \partial \{u>0\} \big) \Big] \cap \Omega.\]  

On the other hand, the \emph{two-phase} subset of \eqref{All-FB} associated with sign-changing minimizers of $\J[\cdot]$ is
\begin{equation} \label{FB2}
    \left( \partial \{u>0\} \cap \partial \{u<0\} \right) \cap \Omega,
\end{equation}
at which the gradient may or may not vanish. More precisely, the subset of \eqref{FB2} 
\begin{equation} \label{FB20}
    \left( \partial \{u>0\} \cap \partial \{u<0\} \right) \cap \{\nabla u = 0 \} \cap \Omega,
\end{equation}
is the \emph{branching} two-phase free boundary which represents prototypically the intersection of \eqref{FB2} with the set $ \overline{\text{Int}\{u=0\}}$. We stress that $\eqref{FB20}$ is contained in \eqref{FB0}. The complement of \eqref{FB20} in \eqref{FB2} is the subset  
\begin{equation} \label{FB2N}
    \Gamma_*:=\left( \partial \{u>0\} \cap \partial \{u<0\} \right) \cap \{\nabla u \neq 0 \} \cap \Omega,
\end{equation}
which is known as the \emph{nonbranching} two-phase free boundary. Such a subset is open relative to $\Gamma$ and may have a neighborhood intersecting $\{u =0\}$ whose $N$-dimensional Lebesgue measure is trivial.

The one-phase version (i.e., nonnegative minimizer) of the minimization problem \eqref{J}-\eqref{Fg} with $p=2$ reduces to the classical Alt--Phillips problem, which was investigated in a series of works \cite{phillips1983minimization,phillips1983hausdoff,phillips1986free}. The particular case of $\gamma =1$, with the linear Laplacian, corresponds to the classical obstacle problem \cite{brezis1974smoothness,caffarelli1977regularity,caffarelli1998obstacle}, which was also studied in the two-phase setting in \cite{uraltseva2001two,shahgholian2004global,shahgholian2006two}. The free boundaries for the two-phase linear problem with $ 1< \gamma < 2$ were analyzed in \cite{fotouhi2017semilinear} and for $0< \gamma < 1$, in dimension two, in \cite{lindgren2008regularity}. A classification of the blowup limits for the planar problem was carried out in \cite{dipierro2023classification}.

In the setting of quasilinear degenerate operators (i.e., $p > 2$), the regularity properties of the solution and the associated free boundary in the one-phase obstacle problem were investigated in \cite{andersson2015optimal,lee2003hausdorff}, whereas those of the two-phase problem in \cite{edquist2009two}, under the restriction that $p$ is close to $2$. As to the range $0 < \gamma < 1$, the interior regularity of nonnegative minimizers was established in \cite{araujo2023sharp}, and for sign-changing minimizers in \cite{leitao2015regularity} in slightly different settings (see also \cite{CdP}, again for $p \sim 2$). The local finiteness of the $(N-1)$-dimensional Hausdorff measure of the free boundary for the one-phase problem in the latter range of $\gamma$ was recently established in \cite{araujo2024geometric}. A treatment of the obstacle one-phase problem in the singular range $1<p<2$ can be found in \cite{karp2000porosity, andersson2015optimal}. The two-phase Bernoulli problem, corresponding to $\gamma = 0$, was studied in \cite{dipierro2018stratification} and a variable-coefficient, one-phase minimization problem of this type in \cite{leitao2015regularity-disc}.

The analysis of the minimization problem of \eqref{J}-\eqref{Fg} poses several challenges. First, the sign-changing nature of a given minimizer precludes identifying explicit interior regularity exponents and, as far as establishing the growth rate away from $\Gamma_0$ is concerned, the use of maximum principles. One is ultimately forced to invoke and compare with the (generally unknown) H\"older exponent of the gradient of $p$-harmonic functions, which amounts to restricting the range of $\gamma$ for which explicit interior regularity or growth rate exponents are ascertained. The second difficulty is the discrepancy in the growth rate of minimizers away from different segments of the free boundaries in the two-phase setting. More precisely, whereas a \emph{linear} growth is expected away from $\Gamma_*$ owing to the interior regularity, the scaling of the problem suggests that a minimizer grows \emph{superlinearly} away from $\Gamma_0$. Such a discrepancy limits the applicability of pointwise gradient estimates (employed, for instance, in \cite{araujo2024geometric}) to neighborhoods outside $\overline{\Gamma}_*$, as they require the scaling-compatible superlinear growth and decay of minimizers. Finally, known approaches for estimating the $(N-1)$-dimensional Hausdorff measure of the free boundaries that rely on differentiating both sides of the associated Euler-Lagrange equation and employing higher derivative bounds (\textit{à la} \cite{lee2003hausdorff}, \cite{edquist2009two}), or estimating the perimeter (e.g., \cite{leitao2015regularity-disc}) are suitable for the special values of $\gamma=1$ or $\gamma=0$, respectively.

\subsection{Main results} In this work, we investigate the minimization problem \eqref{J}-\eqref{Fg} in the entire range of the quasilinear operator $(1<p<\infty)$ and the strength of the nonlinearity $(0 < \gamma <p)$, under no sign assumption on the minimizer. More precisely, we establish interior and pointwise regularity estimates of minimizers as well as local geometric and measure-theoretic properties of the induced free boundaries. Our main results concerning minimizers of $\J[\cdot]$ are the following:

\begin{itemize} 

\item Any local minimizer $u$ of $\J[\cdot]$ is of class $C_{\textnormal{loc}}^{1,\alpha} (\Omega)$, for some $\alpha \in(0,1)$, and satisfies, in the sense of distributions, 
\[\Delta_p u = F'_\gamma(u) \qquad \text{in }\ \{|u|>0\} \cap \Omega,\]
with $F'_\gamma(u) \in L^1_{\textnormal{loc}}(\Omega)$.
\medskip

\item If $x$ is sufficiently close to $\Gamma_0$, then 
\[ |u(x)| \leq C \, \textnormal{dist}\big(x,\Gamma_0\big)^{1+ \min \left\{ \alpha_p - \epsilon, \frac{\gamma}{p-\gamma} \right\}},\]
for any $0< \epsilon \ll 1$. Here, $\alpha_p \in (0,1)$ denotes the optimal H\"older exponent for the gradient of $p$-harmonic functions, and the constant $C >0$ depends only on $\lambda _\pm,N,p,\gamma,$ and $ \|u\|_{L^\infty}$.

\end{itemize}

\textit{A priori} bounds on minimizers are established in \Cref{exist-sec}. The proof of their interior $C^{1,\alpha}$ regularity is based on variational techniques relying mainly on comparison with a $p$-harmonic replacement, as delineated in \Cref{inter-reg-sec}. Such interior regularity is employed in \Cref{growth-sec} to establish the growth rate of minimizers away from the zero-gradient free boundary, $\Gamma_0$, which is dictated by the displayed competition between the optimal regularity for $p$-harmonic functions and the scaling of the problem. Optimality of such a growth rate is, therefore, certified for a restricted range of $\gamma$ that depends on $p$ and $\alpha_p$; see \Cref{restricted}. 

Next, we turn our attention to the free boundaries and establish in \Cref{FB-sec} the following results for any local minimizer $u$ and open, compactly contained subset $\Omega' \Subset \Omega$:  

\begin{itemize} 

\item The set $\Gamma_* \cap \Omega'$ is countably $C^{1,\alpha}$ rectifiable for some $
\alpha \in (0,1)$, and
\[\mathcal{H}^{s}(\Gamma_* \cap \Omega') = 0 \qquad \textnormal{for every} \qquad s > N-1.\]
\item If $ 0 < \gamma < \min\{1,p \alpha_p/(1+\alpha_p) \}$, then 
\[\mathcal{L}^N(\Gamma \cap \Omega') = 0,\] 
and 
\[\mathcal{H}^{N-1+\gamma(p-1)/(p-\gamma)}(\Gamma_0 \cap \Omega') < \infty.\]
\end{itemize}

The local rectifiability and trivial Lebesgue measure of the nonbranching free boundaries, $\Gamma_*$, are a direct consequence of the interior regularity of $u$ and an application of the implicit function theorem; cf. \Cref{zero-lebesgue-non} and \Cref{hausdorff-nonbranch}. On the other hand, the zero local Lebesgue measure of $\Gamma_0$ follows from its local porosity established in \Cref{poro-subsec}. In \Cref{density-subsec}, we develop upper and lower Lebesgue density estimates for $\Gamma_0$. These density estimates, along with a covering argument, are employed to establish the finiteness of the local Hausdorff measure of $\Gamma_0$ in \Cref{hausdorff-sec}.   

\subsection{Notation} We denote by $C$ a positive constant that may vary in numerical value from one occurrence to the next. Occasionally, we use $C(B)$ to emphasize the dependence of $C$ on some quantity $B$. $\R^N$ is equipped with the Euclidean inner product $x\cdot y$ and the induced norm $|x|$. For $x \in \mathbb{R}^N$ and $Y \subset \R^N$, we define the Euclidean distance from $x$ to $Y$ by $\text{dist}(x,Y) := \inf\{|x-y|: y \in Y \}$. Given a set $A$, $\chi_A$ denotes the characteristic function of $A$. The notation $e_i$ signifies the standard $N$-dimensional basis vector with its $i$th entry being one, and $\nu$ an outward unit normal vector with respect to a given surface. By $\mathcal{L}^N$, we shall denote the $N$-dimensional Lebesgue measure and by $\mathcal{H}^s$ the $s$-dimensional Hausdorff measure. The shorthand $\dim_\mathcal{H}(A)$ stands for the Hausdorff dimension of a set $A$ and Card $A$ for its cardinality. $B_r(x_0)$ and $\overline{B}_r(x_0)$ denote, respectively, open and closed, $N$-dimensional, balls with radius $r$ and center $x_0$; their Lebesgue measure is $C(N)r^N$. The shorthand $B_r := B_r(0)$ is frequently used. We will use the notation $(f)_r$ to denote the average of the function $f$ in $B_r$, i.e., $(f)_r:= (\mathcal{L}^N(B_r))^{-1} \int_{B_r} f(x) dx$. Throughout, $p$-(sub/super) harmonicity is understood in the distributional sense. Moreover, $\alpha_p$ stands for the optimal H\"older exponent for the gradient of $p$-harmonic functions. We fix the recurrent quantity $\tau:= \gamma/(p-\gamma)$ for later use. 
 
\section{\textit{A priori} bounds} \label{exist-sec}

By a \emph{minimizer} of $\J[\cdot]$, we mean a function $u \in \Kg(\Omega)$ such that
\begin{equation} \label{Glob-min}
    \J[u] \leq \J[v], \quad \text{for all} \ v \in \Kg(\Omega). 
\end{equation}
\noindent
A minimizer in the sense of \eqref{Glob-min} is automatically a \emph{local minimizer} of $\J[\cdot]$:

\begin{definition} \label{local-min}
      We say $u \in W^{1,p}(\Omega)$ is a local minimizer of $\J[\cdot]$ if $\J[u] \leq \J[u + \varepsilon \varphi]$, for every $\varphi \in W_0^{1,p}(\Omega')$, $\Omega' \Subset \Omega$, and $\varepsilon\in \R$. Alternatively, $\J[u] \leq \J[v]$, for every $v \in W^{1,p}(\Omega')$ satisfying $u - v \in W_0^{1,p}(\Omega')$.
\end{definition}

The existence of at least one minimizer of $\J[\cdot]$ over $\Kg(\Omega)$ is ensured by the direct method of the Calculus of Variations (see \cite[Theorem 3.1]{leitao2015regularity}). An \textit{a priori} bound for such a minimizer in the $W^{1,p}$-norm is obtained in the next proposition.  
 
\begin{proposition} \label{W1p-bound}
    If $u$ is a minimizer of $\J[\cdot]$ over $\Kg(\Omega)$, then   
    \[\|u\|_{W^{1,p}(\Omega)} \leq C,\]
    where the constant $C>0$ depends on $p,\gamma,\lambda_\pm,\Omega$, and $\|g\|_{W^{1,p}(\Omega)}.$ 
\end{proposition}
\begin{proof}
    By the minimality of $u$ in $\Kg(\Omega)$, we have that $\J[u] \leq \J[g]$. Notice, furthermore, that $F_\gamma(u) \geq 0$ and
    \[\int_\Omega (g_\pm)^\gamma dx \leq \mathcal{L}^N(\Omega)^{1-\frac{\gamma}{p}} \|g\|^\gamma_{L^p(\Omega)},\]
    by an application of H\"older's inequality. It follows that 
    \begin{equation} \label{DuLp}
        \|\nabla u\|^p_{L^p(\Omega)} \leq \|\nabla g\|^p_{L^p(\Omega)} + p(\lambda_+ + \lambda_-)\mathcal{L}^N(\Omega)^{1-\frac{\gamma}{p}} \|g\|^\gamma_{L^p(\Omega)}. 
    \end{equation}
    As $u - g \in W_0^{1,p}(\Omega)$, we have, using Poincar\'e's inequality,
    \begin{align} \label{uLp}
        \|u\|_{L^p(\Omega)} &\leq \|u-g\|_{L^p(\Omega)} + \|g\|_{L^p(\Omega)} \nn \\
        & \leq C \|\nabla u-\nabla g\|_{L^p(\Omega)} + \|g\|_{L^p(\Omega)} \nn \\
        & \leq C \|\nabla u\|_{L^p(\Omega)} + (1+C)\|g\|_{W^{1,p}(\Omega)},
    \end{align}
    for some $C = C(p,\Omega)>0$. Combining \eqref{DuLp} and \eqref{uLp} completes the proof.
\end{proof}

One can also demonstrate a local uniform bound for any local minimizer of $\J[\cdot]$, as shown in the following proposition.

 \begin{proposition} \label{uniform-bound}
    If $u$ is a local minimizer of $\J[\cdot]$, then for every $\Omega' \Subset \Omega$,
    \begin{equation*}
        \|u\|_{L^\infty(\Omega')} \leq C_0,
    \end{equation*}
    where $C_0>0$ depends on $N,p,\|u\|_{L^p(\Omega')},$ and $\textnormal{dist}(\Omega',\partial\Omega)$. 
\end{proposition}

\begin{proof}
    It suffices to consider the case $p \leq N$, as one can obtain the bound by Morrey's inequality and \Cref{W1p-bound} otherwise. 
    \\
    For $0 < r < R$, consider $B_r \subset B_R \Subset \Omega$ and a cutoff function $\zeta \in C_0^\infty(B_R)$ such that $0 \leq \zeta \leq   1$ in $B_R$, $\zeta = 1$ in $B_r$ and $|\nabla \zeta| \leq (R-r)^{-1}.$ In addition, for any $\rho>0$ and $k\geq0$, define the set 
    \[A_{k,\rho}:= \{ x \in B_\rho : u(x) > k\}.\]
    Consider the competitor 
    \[w:=u-\zeta (u-k)_+  \leq u,\]
    that differs from $u$ at most in $A_{k,R}$. The local optimality of $u$ yields
    \begin{equation} \label{optim-u-w}
        \int_{A_{k,R}} |\nabla u|^p dx \leq \int_{A_{k,R}} |\nabla w|^p dx + p \int_{A_{k,R}} F_\gamma(w) - F_\gamma(u) dx.
    \end{equation}
    In the set $A_{k,R}$, $w = (1-\zeta)u + \zeta k$ and hence $\nabla w = (1-\zeta) \nabla u - (u-k) \nabla \zeta$. Using the inequality in \Cref{sum-p-harm}, we may estimate 
    \begin{equation*}
        |\nabla w|^p \leq c(p)(1-\zeta)^p|\nabla u|^p + c(p)|\nabla \zeta|^p (u-k)^p.
    \end{equation*}
    Moreover, on the same set,  
    \begin{align*}
        F_\gamma(w) - F_\gamma(u) &= \lambda_+ ( (w_+)^\gamma -(u_+)^\gamma) \\
        &= \lambda_+  \left( \big( (1-\zeta)u + \zeta k \big)^\gamma -u^\gamma \right) \\ 
        &\leq \lambda_+  \left( \big( (1-\zeta)u + \zeta u \big)^\gamma -u^\gamma \right)=0.
    \end{align*}
    In view of the aforementioned estimates and the properties of $\zeta$,  \eqref{optim-u-w} gives 
    \begin{equation} \label{optim-u-w-2}
    \int_{A_{k,r}} |\nabla u|^p dx \leq c(p)\int_{A_{k,R}\backslash A_{k,r}} |\nabla u|^p dx + \dfrac{c(p)}{(R-r)^p} \int_{A_{k,R}} (u-k)^p dx. 
    \end{equation}
    Adding the quantity 
    \[c(p)\int_{A_{k,r}} |\nabla u|^p dx,\]
    to both sides of \eqref{optim-u-w-2} and defining $\theta:= c(p)/(1+c(p)) < 1$, we obtain
     \begin{equation}\label{optim-u-w-3}
    \int_{A_{k,r}} |\nabla u|^p dx \leq \theta \int_{A_{k,R}} |\nabla u|^p dx + \dfrac{C}{(R-r)^p} \int_{A_{k,R}} (u-k)^p dx. 
    \end{equation}
     \Cref{W1p-bound} allows us to  apply \Cref{hole-filling} to \eqref{optim-u-w-3}, whereby we conclude
    \begin{equation} \label{caccip-ineq}
        \int_{B_r} |\nabla (u-k)_+|^p dx \leq \dfrac{C}{(R-r)^p} \int_{B_R} (u-k)_+^p dx. 
    \end{equation}
    With the Caccioppoli inequality \eqref{caccip-ineq} for $u$ and the \textit{a priori} bound given by \Cref{W1p-bound} established, one can apply classical De Giorgi's iteration (e.g., \cite[Theorem 2.1 of Ch. 10]{dibenedetto2023partial}) to conclude that $u$ is essentially bounded above in $B_r$. To obtain a lower bound, one simply notices that $\underline{u}:= -u$ minimizes 
    \[\underline{\J}[v]: =\int_{\Omega} \dfrac{|\nabla v|^p}{p} + \lambda_+ (v_-)^\gamma + \lambda_- (v_+)^\gamma dx.\]
  Therefore, the exact preceding procedure can be used to derive an upper bound for $\underline{u}$ in $B_r$. This completes the proof.  
\end{proof}

We conclude with the following integral growth estimate for the gradient.
 
\begin{proposition} \label{Morrey}
    Suppose $u$ is a local minimizer of $\J[\cdot]$ and $B_{2R} \Subset \Omega$ for some $R>0$. Then
    \begin{equation*}
        \int_{B_R} |\nabla u|^p \leq CR^{N-p},
    \end{equation*}
    for a constant $C>0$ that depends on $N,p,$ and $\|u\|_{L^p(B_{2R})}$.
\end{proposition}

\begin{proof}
    The reasoning is analogous to the proof of the previous proposition. We take a cutoff function $\zeta \in C_0^\infty(B_{2R})$ such that $\zeta = 1$ in $B_R$ and $|\nabla \zeta| \leq R^{-1}$. We set $v:= (1-\zeta)u$ and notice that $F_\gamma(v) = (1-\zeta)^\gamma F_\gamma(u)$. The nonnegativity of $F_\gamma(u)$, therefore, implies that $F_\gamma(v) - F_\gamma(u) \leq 0$ in $B_{2R}$. This, along with the local optimality of $u$, yields
    \begin{align*}
        \int_{B_{2R}} |\nabla u|^p dx \leq \int_{B_{2R}} |\nabla v|^p dx.
    \end{align*}
    Estimating both sides and using the properties of $\zeta$, we obtain 
    \begin{align*}
        \int_{B_{R}} |\nabla u|^p dx \leq c(p) \int_{B_{2R} \backslash B_{R}} |\nabla u|^p dx + c(p) \int_{B_{2R}}|\nabla \zeta|^p |u|^p dx.
    \end{align*}
    Adding  $c(p) \int_{ B_{R}} |\nabla u|^p dx$ to both sides, we deduce for some $\theta < 1$,
    \begin{align*}
        \int_{B_{R}} |\nabla u|^p dx \leq \theta \int_{B_{2R}} |\nabla u|^p dx + \dfrac{C}{R^p} \int_{B_{2R}} |u|^p dx.
    \end{align*}
    Taking into account \Cref{W1p-bound}, \Cref{hole-filling} leads to 
    \begin{equation*}
        \int_{B_R} |\nabla u|^p dx \leq \dfrac{C}{R^p} \|u\|^p_{L^p(B_{2R})} \mathcal{L}^N(B_{2R}) \leq CR^{N-p}.  \end{equation*}  
\end{proof}

\section{Interior regularity} \label{inter-reg-sec}
  
In this section, we establish local regularity for local minimizers of $\J[\cdot]$ and their gradients in a more general setting. More precisely, we consider $\J[\cdot]$ with a $\delta$-scaled potential, that is,
\begin{equation} \label{J-d}
    \Jd[u] := \int_\Omega  \dfrac{|\nabla u|^p}{p} + \delta F_\gamma(u) 
  dx, \ \ \ \ \delta \in [0,\delta_0],
\end{equation}
for a universal and fixed constant $\delta_0>0$ that depends on the data of the problem. We distinguish between the degenerate/linear regime $(2 \leq p < \infty)$ and the singular regime $(1 < p < 2)$. Aside from providing the minimal continuity required to derive the Euler-Lagrange equation and nondegeneracy estimates for minimizers of $\J[\cdot]$, the interior regularity for minimizers of the scaled functional \eqref{J-d} provides the necessary compactness to study the pointwise growth near the free boundaries. 

We begin in the next two subsections by establishing some scaling properties of $\J[\cdot]$ and developing preliminary estimates involving $p$-harmonic replacements.  

\subsection{Scaling invariance} \label{scaling-rmk}

A convenient property of the functional $\J[\cdot]$ is that it preserves minimizers under certain scalings. Assume $R, r>0$ are constants. For another constant $s>0$, the rescaling
\begin{equation*}
    u_s(x):= \dfrac{u(rx+z)}{s},\quad \textnormal{for} \quad x \in B_{R/r}(0), 
\end{equation*}
verifies
\begin{equation*}
    r^{N-p} s^{p} \int_{B_{R/r}(0)}  \dfrac{|\nabla u_s|^p}{p} + r^{p} s^{\gamma -p} F_\gamma(u_s) 
  dy =  \int_{B_{R}(z)}  \dfrac{|\nabla u|^p}{p} + F_\gamma(u) dx.
\end{equation*}
Therefore, if $u$ is a local minimizer of $\J[\cdot]$ in $B_R(z)$, then $u_s$ is a local minimizer of $\J_\delta[\cdot]$ in $B_{R/r}(0)$, with $\delta := r^p s^{\gamma-p}$.
In particular, 
\begin{equation*}
        u_r(x):= \dfrac{u(rx+z)}{r^{1+\tau}}, \quad \textnormal{for} \quad x \in B_{R/r}(0), 
\end{equation*}
satisfies 
\begin{equation*}
     r^{N+p\tau} \int_{B_{R/r}(0)}  \dfrac{|\nabla u_r|^p}{p} + F_\gamma(u_r) 
  dx = \int_{B_{R}(z)}  \dfrac{|\nabla u|^p}{p} + F_\gamma(u) dy.
\end{equation*}
Consequently, if $u$ is a local minimizer of $\J[\cdot]$ in $B_R(z)$, then $u_r$ is a local minimizer of $\J[\cdot]$ in $B_{R/r}(0)$.  

\subsection{$p$-Harmonic replacements and related estimates}

We recall the definition of $p$-harmonic replacements and derive some useful estimates for future use. For simplicity, in this subsection and the ones to come, we work with balls centered at the origin, $B_R := B_R (0)$.   

\begin{definition} We say that $\bar{u} \in W^{1,p}(B_R)$ is a $p$-harmonic replacement of a function $u \in W^{1,p}(B_R)$ if $\bar{u}$ is $p$-harmonic in $B_R$ and $u - \bar{u} \in W_0^{1,p}(B_R)$.
\end{definition}

Since $\bar{u}$ is $p$-harmonic, it is a minimizer of 
\begin{equation*}
    \int_{B_R} \dfrac{|\nabla v|^p}{p} dx,
\end{equation*}
and, as such, it satisfies (\cite[Theorem 2.7]{lindqvist2019notes})
\begin{equation} \label{ubar-leq-u}
        \int_{B_R} |\nabla \bar{u}|^p dx \leq \int_{B_R} |\nabla u|^p dx.
    \end{equation}
    
Moreover, by the maximum/minimum principle for $p$-harmonic functions, we have 
\begin{equation*}
        \osc_{B_R} \bar{u} \leq \osc_{\partial B_R} u,
    \end{equation*}
from which one straightforwardly deduces
\begin{equation}\label{pharmonic-osc-2}
    \sup_{B_R} |\bar{u}|^\gamma \leq \left(\sup_{\partial B_R} |u|\right)^\gamma \leq  \left(\sup_{ B_R} |u|\right)^\gamma,
\end{equation}
for any $\gamma>0$. The next lemma demonstrates an auxiliary inequality.

\begin{lemma} \label{two-cases-estimate-p-lemma}
Let $u \in W^{1,p}(B_{2R})$ and suppose $\bar{u} \in W^{1,p}(B_{2R})$ is its  $p$-harmonic replacement in $B_{2R}$. Then, there exist constants $C=C(N,p)>0$ such that
\begin{equation} \label{two-cases-p}
    \int_{B_{2R}} |\nabla u|^p - |\nabla {\bar{u}}|^p dx \geq  
\left\{
\begin{array}{lcl}
        C  \displaystyle \int_{B_{2R}}  \left( |\nabla u| + |\nabla {\bar{u}}| \right)^{p-2} |\nabla u - \nabla \bar{u}|^2 dx & \textnormal{if} & 1 < p \leq 2 \\
        \\
        C \displaystyle \int_{B_{2R}} |\nabla u - \nabla \bar{u}|^p dx &  \textnormal{if} & p \geq 2.
\end{array}
\right.
\end{equation}
\end{lemma}

\begin{proof}
The proof of these inequalities is rather routine. Note that the $p$-harmonicity of $\bar{u}$ implies 
    \begin{equation*}
        \int_{B_{2R}}  |\nabla \bar{u}|^{p-2} \nabla \bar{u}\cdot \nabla(u - \bar{u}) dx =0.
    \end{equation*}
    For $0 \leq s \leq 1$, set $u_s := su + (1-s)\bar{u}$ and compute
    \begin{align*}
        \int_{B_{2R}} |\nabla u|^p - |\nabla {\bar{u}}|^p dx &= \int_0^1 \dfrac{d}{ds}\left( \int_{B_{2R}} |\nabla u_s|^p dx\right) ds \\
        &= p \int_0^1 ds \int_{B_{2R}} |\nabla u_s|^{p-2} \nabla u_s \cdot \nabla(u - \bar{u}) dx \\
        &= p \int_0^1 ds \int_{B_{2R}} \left( |\nabla u_s|^{p-2} \nabla u_s - |\nabla \bar{u}|^{p-2} \nabla \bar{u} \right) \cdot \nabla(u - \bar{u}) dx\\
        &= p \int_0^1 \dfrac{ds}{s} \int_{B_{2R}} \left( |\nabla u_s|^{p-2} \nabla u_s - |\nabla \bar{u}|^{p-2} \nabla \bar{u} \right) \cdot \nabla(u_s - \bar{u}) dx,
    \end{align*}
where we used the observation that 
\begin{equation} \label{sus}
    s(u - \bar{u}) = u_s - \bar{u}.
\end{equation}
Recalling the elementary inequalities (cf. \cite[Ch. 12] {lindqvist2019notes}), valid for $a,b \in \mathbb{R}^N$,
\begin{equation*}
    \left(|a|^{p-2} a - |b|^{p-2} b \right) \cdot (a-b)\geq C(p) \begin{cases}
        \left(|a| + |b| \right)^{p-2} |a-b|^2 \quad \textnormal{if} \quad 1 < p \leq 2 \\
        |a-b|^p \quad \quad \quad \quad \quad  \quad \ \ \textnormal{if} \quad p \geq 2,
    \end{cases}
\end{equation*}
we obtain, for $1 < p \leq 2$, 
\begin{align*}
    \int_{B_{2R}} |\nabla u|^p - |\nabla {\bar{u}}|^p dx &\geq C \int_0^1 \dfrac{ds}{s}\int_{B_R} \left(|\nabla u_s|+ |\nabla \bar{u}| \right)^{p-2} |\nabla u_s - \nabla \bar{u}|^2 dx \\
    &\geq C
    \int_0^1 s ds \int_{B_{2R}} \left(|\nabla u|+ |\nabla \bar{u}| \right)^{p-2} |\nabla u - \nabla \bar{u}|^2 dx \\
    &= C \int_{B_{2R}} \left(|\nabla u|+ |\nabla \bar{u}| \right)^{p-2} |\nabla u - \nabla \bar{u}|^2 dx,
\end{align*}
where the second inequality is due to \eqref{sus}, $p-2 \leq 0$, and the bound
\begin{equation*}
    |\nabla u_s|+ |\nabla \bar{u}| \leq s|\nabla u| + (2-s) |\nabla \bar{u}| \leq 2 \left(|\nabla u| + |\nabla \bar{u}| \right),
\end{equation*}
as desired. The case for $p\geq 2$ is analogous; see \cite[Lemma 2.4]{leitao2015regularity}.
\end{proof}

The following lemma provides a practical estimate for the difference of the nonlinearity appearing in $\Jd[\cdot]$. 

 \begin{lemma} \label{nonlinearity-estim-lemma} If $u \in W^{1,p}(B_{2R})$, and $\bar{u} \in W^{1,p}(B_{2R})$ is its $p$-harmonic replacement in $B_{2R} \Subset \Omega$, then
    \begin{equation*}
    \int_{B_{2R}}  F_\gamma(\bar{u}) - F_\gamma(u) dx \leq \begin{cases}
        C \displaystyle \int_{B_{2R}} | u - \bar{u} |^\gamma  dx \quad \textnormal{if} \quad 0 < \gamma \leq 1 \\
        \\
        C \displaystyle  \int_{B_{2R}}|u - \bar{u}| dx \ 
         \quad \: \textnormal{if} \quad   1 \leq \gamma < p ,
    \end{cases}
\end{equation*}
for constants $C>0$ depending on $\lambda_\pm,\gamma,$ and $\|u\|_{L^\infty(B_{2R})}$.  
\end{lemma}

\begin{proof}
The case $0 <\gamma \leq 1$ was obtained in the proof of \cite[Theorem 1.1]{leitao2015regularity}, with a constant $C=C(\lambda_\pm)>0$. We therefore focus on the case $1 < \gamma < p$. One first notices that 
\begin{equation*} \label{nonlin-diff}
    \left| F_\gamma(\bar{u}) - F_\gamma(u) \right| \leq \lambda_+  \left| (\bar{u}_+ )^{\gamma} - (u_+)^\gamma \right| + \lambda_-  \left| (\bar{u}_- )^{\gamma} - (u_-)^\gamma \right|.
\end{equation*}
Then, the mean value theorem applied to the map $s \mapsto s^\gamma$ yields
\begin{equation*} \label{mvt}  
    \left| (\bar{u}_\pm)^\gamma - (u_\pm)^\gamma \right| \leq \gamma (\bar{u}_\pm+ u_\pm)^{\gamma-1} \left| \bar{u}_\pm - u_\pm \right|, 
\end{equation*}
and, by \eqref{pharmonic-osc-2}, we readily estimate
\begin{equation*} \label{max-prin-bound}
     (\bar{u}_\pm+ u_\pm)^{\gamma-1} \leq \left(2 \sup_{B_{2R}} |u|\right)^{\gamma-1}.
\end{equation*}
The fact that both the positive and negative parts of $u$ are $1$-Lipschitz functions implies
\begin{equation*} 
    \left| F_\gamma(\bar{u}) - F_\gamma(u) \right| \leq C |\bar{u} - u|.
\end{equation*}
for a constant $C>0$ that depends on $\lambda_\pm,\gamma$, and $\|u\|_{L^\infty(B_{2R})}$.
\end{proof}

\subsection{The degenerate regime}

Here, we aim at establishing H\"older regularity for the gradient of minimizers of $\Jd[\cdot]$ in the range $p\geq 2$. The following proposition is the essential ingredient to derive such regularity. 

\begin{proposition}
 \label{grad-u-2-estim}
    Suppose $u$ is a local minimizer of $\J_\delta[\cdot]$ with $0\leq \delta \leq \delta_0$ and $2 \leq p < \infty$. Let $B_{4R} \Subset \Omega$ be arbitrary and $\bar{u} \in W^{1,p}(B_{2R})$ be the $p$-harmonic replacement of $u$ in $B_{2R}$. Then, there exist constants $C>0$, depending on $\delta_0, \lambda _\pm, N, p, \gamma$ and $\|u\|_{L^\infty(B_{4R})}$, such that 
    \begin{equation*}
        \int_{B_R} |\nabla u - \nabla \bar{u}|^2 dx \leq \begin{cases}
            C \displaystyle R^{N+ \tfrac{2\gamma}{p-\gamma}} \quad \textnormal{if} \quad 0 < \gamma \leq 1 \\ \\
            C\displaystyle R^{N+ \tfrac{2}{p-1}}  \quad    \textnormal{if} \quad 1 \leq \gamma < p.
        \end{cases}
    \end{equation*}
\end{proposition}
\begin{proof}
For $p\geq2$, we have by \Cref{two-cases-estimate-p-lemma} that 
\begin{equation} \label{degen-diff}
    \int_{B_{2R}} |\nabla u - \nabla \bar{u}|^p dx \leq \int_{B_{2R}} |\nabla u|^p - |\nabla {\bar{u}}|^p dx.
\end{equation}
On the other hand, the optimality of $u$ and the bounds granted by \Cref{nonlinearity-estim-lemma} yield
\begin{align} \label{nonlinearity-estim}
        \displaystyle\int_{B_{2R}} |\nabla u|^p - |\nabla {\bar{u}}|^p dx &\leq \delta_0 p \int_{B_{2R}} F_\gamma(\bar{u}) - F_\gamma(u)    dx \nn \\
        &\leq \begin{cases}
        C\displaystyle\int_{B_{2R}} |u - \bar{u}|^\gamma  dx \quad \textnormal{if} \quad 0 < \gamma \leq 1 \\ \\
        C \displaystyle \int_{B_{2R}}|u - \bar{u}|   dx  
         \quad \ \:  \textnormal{if} \quad   1 \leq \gamma < p,
    \end{cases}
\end{align}
which, after an application of H\"older's inequality, reads
\begin{align} \label{nonlinearity-estim-2}
        \int_{B_{2R}} |\nabla u|^p - |\nabla {\bar{u}}|^p dx &\leq \begin{cases}
            C\displaystyle \mathcal{L}^N(B_{2R})^{1-  \frac{\gamma}{p}} \left(\int_{B_{2R}} |u-\bar{u}|^p dx \right)^{\frac{\gamma}{p}} \quad \textnormal{if} \quad 0 < \gamma \leq 1\\ \\
            C \displaystyle \mathcal{L}^N(B_{2R})^{1-  \frac{1}{p}} \left(\int_{B_{2R}} |u-\bar{u}|^p dx \right)^{\frac{1}{p}} \quad \textnormal{if} \quad 1 \leq \gamma < p.
        \end{cases}
\end{align}
In view of $\eqref{degen-diff}$ and \eqref{nonlinearity-estim-2}, we deduce
\begin{equation}  \label{pth-grad-estim-deg}
    \int_{B_{2R}} |\nabla u - \nabla \bar{u}|^p dx \leq \begin{cases}
        C \displaystyle R^{N \left( 1-  \frac{\gamma}{p} \right) } \left( \int_{B_{2R}} |u - \bar{u}|^p dx \right)^{\frac{\gamma}{p}}  \quad \textnormal{if} \quad 0 < \gamma \leq 1 \\ \\

        C \displaystyle R^{N \left( 1-  \frac{1}{p} \right) } \left( \int_{B_{2R}} |u - \bar{u}|^p dx \right)^{\frac{1}{p}}   \quad \textnormal{if}  \quad  1 \leq \gamma < p,
    \end{cases} 
\end{equation}
for constants $C>0$ that may depend on $\lambda_\pm,N,p,\gamma, \|u\|_{L^\infty(B_{4R})}$ and $\delta_0$. Since $u - \bar{u} \in W_0^{1,p}(B_{2R})$, the integrals on the right-hand side of \eqref{pth-grad-estim-deg} can be estimated via Poincar\'e's inequality 
\begin{equation} \label{pincare}
    \int_{B_{2R}} |u - \bar{u}|^p dx \leq C(N,p) R^p  \int_{B_{2R}} |\nabla u - \nabla \bar{u}|^p dx,
\end{equation}
and, upon simplification, we obtain
\begin{equation} \label{grad-comparison-degen}
    \int_{B_{2R}} |\nabla u - \nabla \bar{u}|^p dx \leq \begin{cases}
        C \displaystyle R^{N +\tfrac{p\gamma}{p-\gamma}}   \quad \textnormal{if} \quad 0 < \gamma \leq 1 \\ \\
        C \displaystyle R^{N + \tfrac{p}{p-1}}    \quad   \textnormal{if}  \quad  1 \leq \gamma < p.
    \end{cases}
\end{equation}
Combining \eqref{grad-comparison-degen} with the basic inequality 
\begin{equation*}
    \int_{B_{2R}} |\nabla u - \nabla \bar{u}|^2 dx \leq \mathcal{L}^N(B_{2R})^{1-  \frac{2}{p}} \left(\int_{B_{2R}} |\nabla u - \nabla \bar{u}|^p dx \right)^{\frac{2}{p}},
\end{equation*}
completes the proof.
\end{proof}

Before proving our main result in this subsection, we recall the following oscillation estimate for the gradient of $p$-harmonic functions.

\begin{lemma} [{{\cite[Theorem 5.1] {dibenedetto1993higher}}}] \label{harm-grad-osc-estim}  If $\bar{u}$ is a $p$-harmonic function in $B_R$, then there exist constants $\alpha_p \in (0,1)$ and $c>1$, both depending on $N$ and $p$, such that for all $0 \leq r \leq R$, 
\begin{equation*}
    \int_{B_r} |\nabla \bar{u} - (\nabla \bar{u})_{r}|^2 dx \leq c \left(\dfrac{r}{R} \right)^{N+2\alpha_p} \int_{B_R} |\nabla \bar{u} - (\nabla \bar{u})_{R}|^2 dx. 
\end{equation*}
\end{lemma}

We are now in a position to demonstrate the following theorem.

\begin{theorem} \label{interior-reg}
    If $u$ is a local minimizer of $\Jd[\cdot]$ with $0 \leq \delta \leq \delta_0$ and $2 \leq p < \infty$, then $u \in C^{1,\alpha}_{\textnormal{loc}}(\Omega)$ for some $\alpha \in (0,1)$. Moreover, there exists a constant $C_1>0$, depending on  $\lambda _\pm,N,p,\gamma,\delta_0, \|u\|_{L^\infty(\Omega')}$, and $\textnormal{dist}(\Omega',\partial \Omega)$, such that 
     \begin{equation}
    \|u\|_{C^{1,\alpha}(\Omega')} \leq C_1.
    \end{equation}
\end{theorem}
\begin{proof}
Take $B_{4R} \Subset \Omega$ and let $\bar{u}$ be the $p$-harmonic replacement of $u$ in $B_{2R}$. Estimating via \Cref{sum-p-harm} and \Cref{harm-grad-osc-estim}, we readily have, for $0 \leq r \leq R$,
\begin{align} \label{grad-u-avg-estim1}
    \int_{B_r} |\nabla u - (\nabla u)_{r}|^2 dx &\leq C \int_{B_r} |\nabla \bar{u} - (\nabla \bar{u})_{r}|^2 dx \nn \\
    & \quad + C \int_{B_r} |\nabla u - \nabla \bar{u}|^2 dx + C \int_{B_r} | (\nabla \bar{u})_{r}- (\nabla u)_{r}|^2 dx \nn \\
    &\leq C(N,p) \left(\dfrac{r}{R} \right)^{N+2\alpha_p} \int_{B_R} |\nabla \bar{u} - (\nabla \bar{u})_{R}|^2 dx \nn \\
    & \quad + C(N) \int_{B_r} |\nabla u - \nabla \bar{u}|^2 dx.
\end{align}
One can further rewrite the first term on the right-hand side of \eqref{grad-u-avg-estim1} using the estimate
\begin{align*}
    \int_{B_R} |\nabla \bar{u} - (\nabla \bar{u})_{R}|^2 dx &\leq C \int_{B_R} |\nabla u - (\nabla u)_{R}|^2 dx \nn  \\ 
    & \quad + C \int_{B_R} |\nabla \bar{u} - \nabla u|^2 dx + C \int_{B_R} | (\nabla u)_{R}-(\nabla \bar{u})_{R}|^2 dx \nn \\
    &\leq C \int_{B_R} |\nabla u - (\nabla u)_{R}|^2 dx + C(N) \int_{B_R} |\nabla u - \nabla \bar{u}|^2 dx. 
\end{align*}
Substituting the above inequality into \eqref{grad-u-avg-estim1} and applying \Cref{grad-u-2-estim}, we arrive at 
\begin{align*} \label{comparison-estimate}
\int_{B_r} |\nabla u - (\nabla u)_{r}|^2 dx & \leq C \left[\left(\dfrac{r}{R} \right)^{N+2\alpha_p} \int_{B_R} |\nabla u - (\nabla u)_{R}|^2 dx + \int_{B_R} |\nabla u - \nabla \bar{u}|^2 dx \right]  \nn  \\
    & \leq A\left(\dfrac{r}{R} \right)^{N+2\alpha_p} \int_{B_R} |\nabla u - (\nabla u)_{R}|^2 dx \nn \\
    &\quad + B\begin{cases}
            \displaystyle R^{N+ \tfrac{2\gamma}{p-\gamma}}  \quad \textnormal{if} \quad  0 < \gamma \leq 1 \\ \\
            \displaystyle
            R^{N+ \tfrac{2}{p-1}} \quad \textnormal{if} \quad   1 \leq \gamma < p,
        \end{cases}
\end{align*}
for some constants $A = A(N,p)>0$ and $B= B(\lambda _\pm,N,p,\gamma, \|u\|_{L^\infty(B_{4R})},\delta_0)>0$. Setting
     \begin{equation*}
        \phi(s):= \int_{B_s} |\nabla u - (\nabla u)_{s}|^2 dx,\end{equation*}
     and applying \Cref{Algebraic-ineq}, we conclude 
    \begin{equation*}
    \phi(r) \leq Cr^{N+2 \alpha},
    \end{equation*}
    for all $0 \leq r \leq R$, where 
    \begin{equation*}
        \alpha := 
        \begin{cases}
        \tfrac{\gamma}{p-\gamma} \quad \quad \textnormal{if}  \quad 0 < \gamma \leq 1 \quad \textnormal{and}  \quad \tfrac{\gamma}{p-\gamma} < \alpha_p \\
        \alpha_p - \epsilon  \quad   \textnormal{if}  \quad 0 < \gamma \leq 1 \quad \textnormal{and}  \quad \tfrac{\gamma}{p-\gamma} 
        \geq \alpha_p \\
        \tfrac{1}{p-1} \quad \quad    \textnormal{if}  \quad 1 \leq \gamma < p \quad \textnormal{and} \quad  \tfrac{1}{p-1} < \alpha_p \\
        \alpha_p - \epsilon \quad    \textnormal{if}  \quad 1 \leq \gamma < p \quad \textnormal{and} \quad  \tfrac{1}{p-1} \geq \alpha_p ,
        \end{cases}
    \end{equation*}
for any $0 < \epsilon \ll 1$. By Campanato's characterization (\cite[Theorem 1.2 of Ch. III] {giaquinta1983multiple}), $\nabla u$ is $\alpha$-H\"older continuous locally in $\Omega$. The dependencies of the constant $C_1>0$ follow from the definition of Campanato's norm and the constants defined above. This completes the proof 
\end{proof}

\subsection{The singular regime} We now turn our attention to the regularity of minimizers in the singular range, that is, $1 < p < 2$. We invoke the following nonlinear map to be employed with the gradient: for an arbitrary vector $a\in \mathbb{R}^N$, we consider 
\begin{equation*}
    V(a):= |a|^{\frac{p-2}{2}}a,
\end{equation*}
which satisfies (\cite[page 240]{duzaar2004p})
\begin{align} \label{V-ineq}
    c^{-1} (|a|^2+|b|^2)^{\frac{p-2}{2}} |a-b|^2 &\leq |V(a)-V(b)|^2 \leq c (|a|^2+|b|^2)^{\frac{p-2}{2}} |a-b|^2,
\end{align}
for some $c=c(N,p)\geq1$. We first establish the counterpart of \Cref{grad-u-2-estim} for the singular range.

\begin{proposition}
 \label{grad-u-2-estim-sing}
    Suppose $u$ is a local minimizer of $\J_\delta[\cdot]$ with $0\leq \delta \leq \delta_0$ and $1 < p < 2$. Let $B_{4R} \Subset \Omega$ be arbitrary and $\bar{u} \in W^{1,p}(B_{2R})$ be the $p$-harmonic replacement of $u$ in $B_{2R}$. Then, there exist constants $C>0$, depending on $\delta_0, \lambda _\pm, N, p, \gamma$ and $\|u\|_{L^\infty(B_{4R})}$, such that 
    \begin{equation*}
        \int_{B_R} \big|V(\nabla u) - V(\nabla \bar{u})\big|^2 dx \leq \begin{cases}
            C \displaystyle R^{N+\tfrac{p\gamma}{2-\gamma}} \ \ \quad \textnormal{if} \quad 0 < \gamma \leq 1 \\ \\
            C\displaystyle R^{N+p}  \quad \quad  \quad  \textnormal{if} \quad 1 \leq \gamma < p.
        \end{cases}
    \end{equation*}
\end{proposition}

\begin{proof}
    For any $B_{4R} \Subset \Omega$, let $\bar{u}$ be the $p$-harmonic replacement of $u$ in $B_{2R}$. Applying \eqref{V-ineq} with $a:= \nabla u$ and $b:=\nabla \bar{u}$, we deduce
    \begin{equation}
        |\nabla u - \nabla \bar{u}|^p \leq C(N,p)  \big( |\nabla u|^2 + |\nabla \bar{u}|^2 \big)^{\frac{p(2-p)}{4}} \big| V(\nabla u) - V(\nabla \bar{u}) \big|^p,
    \end{equation}
    while an application of H\"older's inequality provides
    \begin{align} \label{sing-holder}
    \int_{B_{2R}} |\nabla u - \nabla \bar{u}|^p dx &\leq C(N,p) \left( \int_{B_{2R}}  \big( |\nabla u|^2 + |\nabla {\bar{u}}|^2 \big)^{\frac{p}{2}} dx \right)^{1-\frac{p}{2}} \nn \\
    & \quad \times \left( \int_{B_{2R}} \big|V(\nabla u) - V(\nabla \bar{u})\big|^2 dx \right)^\frac{p}{2}.
\end{align}
In view of \Cref{sum-p-harm}, \eqref{ubar-leq-u}, and \Cref{Morrey}, we have
\begin{equation*}
    \left( \int_{B_{2R}}  \big( |\nabla u|^2 + |\nabla {\bar{u}}|^2 \big)^{\frac{p}{2}} dx \right)^{1-\frac{p}{2}} \leq C  \left( \int_{B_{2R}}  |\nabla u|^p dx \right)^{1-\frac{p}{2}} \leq C R^{N\left(1-\frac{p}{2} \right) + p\left( \frac{p}{2}-1 \right)},
\end{equation*}
and thus, \eqref{sing-holder} simplifies to the instrumental estimate 
\begin{equation} \label{sing-holder-2}
    \int_{B_{2R}} |\nabla u - \nabla \bar{u}|^p dx \leq CR^{N\left(1-\frac{p}{2} \right) + p\left( \frac{p}{2}-1 \right)} \left( \int_{B_{2R}} \big|V(\nabla u) - V(\nabla \bar{u})\big|^2 dx \right)^\frac{p}{2}.
\end{equation}
Recall that, as a consequence of the minimality of $u$ and \Cref{nonlinearity-estim-lemma}, we have obtained, via \eqref{nonlinearity-estim},\eqref{nonlinearity-estim-2}, and \eqref{pincare}, the estimates
\begin{align} \label{nonlinearity-estim-3}
        \int_{B_{2R}} |\nabla u|^p - |\nabla {\bar{u}}|^p dx &\leq \begin{cases}
            C\displaystyle R^{N\left(1-  \frac{\gamma}{p} \right)+\gamma} \left(\int_{B_{2R}} |\nabla u-\nabla \bar{u}|^p dx \right)^{\frac{\gamma}{p} } \quad \textnormal{if} \quad 0 < \gamma \leq 1\\ \\
            C \displaystyle R^{N\left(1-  \frac{1}{p} \right) +1} \left(\int_{B_{2R}} |\nabla u-\nabla \bar{u}|^p dx \right)^{\frac{1}{p}} \quad \textnormal{if} \quad 1 \leq \gamma < p,
        \end{cases}
\end{align}
which, in view of \eqref{sing-holder-2}, read
\[\int_{B_{2R}} |\nabla u|^p - |\nabla {\bar{u}}|^p dx\]
\begin{equation} \label{nonlinearity-estim-4}
         \leq \begin{cases}
            C\displaystyle R^{N\left(1-\frac{\gamma}{2} \right)+\frac{p\gamma}{2}} \left(\int_{B_{2R}} \big|V(\nabla u) - V(\nabla \bar{u})\big|^2 dx \right)^{\frac{\gamma}{2} } \quad \: \textnormal{if} \quad 0 < \gamma \leq 1 \\ \\
            C \displaystyle R^{\frac{N+p}{2}} \left(\int_{B_{2R}} \big|V(\nabla u) - V(\nabla \bar{u})\big|^2 dx \right)^{\frac{1}{2}} \quad \quad \ \quad  \ \ \ \textnormal{if} \quad 1 \leq \gamma < p.
        \end{cases}
\end{equation}
The next step is to estimate \eqref{nonlinearity-estim-4} from below. To this end, we notice that upon combining   \eqref{V-ineq} and the elementary inequality for $1 < p < 2$,  
    \begin{equation*}
        (|a|+|b|)^{p-2} \leq (|a|^2+|b|^2)^{\frac{p-2}{2}} \leq 2^\frac{2-p}{2} (|a|+|b|)^{p-2},
    \end{equation*}
    with the relevant inequality in \Cref{two-cases-estimate-p-lemma} and \eqref{nonlinearity-estim-4}, we arrive at 
    \begin{equation*}
        \int_{B_{2R}} |V(\nabla u) - V(\nabla \bar{u})|^2 dx  \leq  \begin{cases}
            C\displaystyle R^{N+\tfrac{p\gamma}{2-\gamma}}  \quad \: \textnormal{if} \quad 0 < \gamma \leq 1 \\ \\
            C \displaystyle R^{N + p} \quad \ \quad \: \textnormal{if} \quad 1 \leq \gamma < p,
        \end{cases}
    \end{equation*}
    which completes the proof. 
\end{proof}
 
With the previous estimate, the desired gradient regularity follows at once.

\begin{theorem} \label{interior-reg-sing}
    If $u$ is a local minimizer of $\Jd[\cdot]$ with $0 \leq \delta \leq \delta_0$ and $1 < p <2$, then $u \in C^{1,\alpha}_{\textnormal{loc}}(\Omega)$ for some $\alpha \in (0,1)$. Moreover, there exists a constant $C_1>0$, depending on  $\lambda _\pm,N,p,\gamma,\delta_0, \|u\|_{L^\infty(\Omega')}$, and $\textnormal{dist}(\Omega',\partial \Omega)$, such that 
     \begin{equation}
    \|u\|_{C^{1,\alpha}(\Omega')} \leq C_1.
    \end{equation}
\end{theorem}

\begin{proof}
Take $B_{4R} \Subset \Omega$ and let $\bar{u}$ be the $p$-harmonic replacement of $u$ in $B_{2R}$. By \cite[Proposition 3]{duzaar2004p}, \Cref{harm-grad-osc-estim} also holds if $\nabla \bar{u}$ is replaced by the quantity $V(\nabla \bar{u})$. Thus, in view of \Cref{grad-u-2-estim-sing} and following the same reasoning as in the proof of \Cref{interior-reg}, one deduces, for all $0 \leq r \leq R$,
    \begin{align*}
        \int_{B_r}  \big|V(\nabla u) - V(\nabla u)_r\big|^2 dx &\leq  A\left(\dfrac{r}{R} \right)^{N+2\alpha_p} \int_{B_R} \big|V(\nabla u) - V(\nabla u)_R\big|^2 dx \nn \\
    & \quad + B \begin{cases}
            \displaystyle R^{N+\tfrac{p\gamma}{2-\gamma}}  \quad \ \ \textnormal{if} \quad  0 < \gamma \leq 1 \\ \\
            \displaystyle
            R^{N+p} \quad \quad \quad   \textnormal{if} \quad    1 \leq \gamma < p,
        \end{cases}  
    \end{align*}
    for some constants $A,B>0$. As in the conclusion of the proof of \Cref{interior-reg}, we set 
     \begin{equation*}
        \phi(s):= \int_{B_s} \big|V(\nabla u) - V(\nabla u)_s\big|^2 dx,\end{equation*}
     and apply \Cref{Algebraic-ineq} to conclude 
    \begin{equation*}
    \phi(r) \leq Cr^{N+2 \alpha},
    \end{equation*}
    for all $0 \leq r \leq R$, where  
    \begin{equation*}
        \alpha := 
        \begin{cases} 
        \tfrac{p\gamma}{2(2-\gamma)} \quad  \ \ \textnormal{if}  \quad 0 < \gamma \leq 1 \quad \textnormal{and}  \quad \tfrac{p\gamma}{2(2-\gamma)} < \alpha_p \\
        \alpha_p - \epsilon  \quad \: \    \textnormal{if}  \quad 0 < \gamma \leq 1 \quad \textnormal{and}  \quad \tfrac{p\gamma}{2(2-\gamma)} 
        \geq \alpha_p \\
        \tfrac{p}{2} \quad \quad    \quad \  \: \textnormal{if}  \quad 1 \leq \gamma < p \quad \textnormal{and} \quad \ \ \quad  \ \tfrac{p}{2} < \alpha_p \\
        \alpha_p - \epsilon \quad \: \    \textnormal{if}  \quad 1 \leq \gamma < p \quad \textnormal{and} \quad \ \ \quad \  \tfrac{p}{2} \geq \alpha_p ,
        \end{cases}
    \end{equation*}
for any $0 < \epsilon \ll 1$. Hence, $V(\nabla u)$ is locally $\alpha$-H\"older continuous and so is $\nabla u$ (cf. \cite[Lemma 3]{duzaar2004p}).
\end{proof}

We conclude this section with the following straightforward consequence of the continuity of minimizers for $\J[\cdot]$: the Euler-Lagrange equation in the noncoincidence sets. Recall that $F'_\gamma$ is given by \eqref{RHS-F}.

\begin{theorem} \label{Dist-sol}
    A local minimizer $u$ of $\J[\cdot]$ satisfies 
    $\Delta_pu_\pm =F'_\gamma(u_\pm)$ in $\{ u_\pm>0\} \cap \Omega$ in the sense of distributions.  
\end{theorem}

 \begin{proof}
     We display the argument on $\{u_+>0\}$, the reasoning for the negative phase being similar. For any $x_0 \in \{u_+>0\}$, the continuity of $u_+$ implies that $u_+ \geq u_+(x_0)/2 >0$ in some neighborhood $\Omega_{x_0} \Subset \{u_+>0\}$. Consequently, $F'_\gamma(u_+) = \gamma \lambda_+ u_+^{\gamma-1}$ is bounded therein for any $\gamma>0$. Taking $\varphi \in C_0^\infty(\Omega_{x_0})$ and $s \in \R$ such that $|s|$ is sufficiently small, we have that $u_+ + s \varphi \geq u_+(x_0)/4>0$ in $\Omega_{x_0}$. The optimality of $u_+$, therefore, yields
     \begin{equation*} 
        0 = \dfrac{d}{ds} \J[u_+ + s\varphi] \Big|_{s=0} = \int_{ \{u_+>0\} \cap \Omega}      |\nabla u_+|^{p-2} \nabla u_+ \cdot \nabla \varphi  + \gamma \lambda_+u_+^{\gamma-1}\varphi dx.
\end{equation*}
 \end{proof}

The local integrability of $F_\gamma'(u)$ is demonstrated in the following proposition and is employed in due course to obtain an upper Lebesgue density estimate for the zero-gradient free boundaries. 

 \begin{proposition} \label{Loc-integ}
     $F'_\gamma(u) \in L^1_{\textnormal{loc}}(\Omega)$.
 \end{proposition}
 \begin{proof}
     We proceed by approximation as in \cite{phillips1983hausdoff}. Consider a function $\psi \in C^\infty(\R)$  satisfying $\psi(t)= 0$ if $t\leq 1$ and $\psi(t)= 1$ if $t\geq 2$ with $\psi'(t)\geq 0$. For any $\delta>0$, define $\psi_\delta(t) := \psi(t/\delta)$. Fix $\Omega' \Subset \Omega$ and take a cutoff function $\xi \in C_0^\infty(\Omega)$ such that $0 \leq \xi \leq 1$ in $\Omega$ and $\xi = 1$ in $\Omega'$. From the integral identity of \Cref{Dist-sol} with $ \varphi=\psi_\delta(u_\pm) \xi $ as a test function, one deduces
     \begin{align*}
         \int_{ \{u_\pm>0\} \cap \Omega}\gamma \lambda_\pm u_\pm^{\gamma-1}  \psi_\delta(u_\pm) \xi  dx &= - \int_{ \{u_\pm>0\} \cap \Omega}      |\nabla u_\pm|^p \psi_\delta'(u_\pm) \xi  dx \\
         & \ \ \ - \int_{ \{u_\pm>0\} \cap \Omega}      |\nabla u_\pm|^{p-2} \nabla u_\pm  \cdot  \psi_\delta(u_\pm) \nabla \xi  dx  \\
         & \leq \int_{ \{u_\pm>0\} \cap \supp{\nabla \xi}}      |\nabla u_\pm|^{p-1} |\nabla \xi|  dx,
     \end{align*}
    since $\xi$ and $\psi_\delta'$ are nonnegative. The last term is obviously finite due to \Cref{W1p-bound}. By the monotone convergence theorem, we may send $\delta \to 0$ so that $\psi_\delta(u_\pm) \to \chi_{\{u_\pm>0\}}$ and conclude $\gamma \lambda_\pm u_\pm^{\gamma-1} \chi_{\{u_\pm>0\}} \in L^1_{\textnormal{loc}}(\Omega)$.
 \end{proof}

\section{Growth and decay estimates}  \label{growth-sec}

In this section, we establish growth and nondegeneracy estimates for minimizers of $\J[\cdot]$ at the free boundaries.  

\subsection{Growth rate} The following theorem establishes the growth rate of minimizers away from $\Gamma_0$, comprising the one-phase (positive or negative) as well as the two-phase branching free boundaries. An illustration of optimality is provided thereafter.

\begin{theorem}\label{Optimal-growth-thm}
For any local minimizer $u$ of $\J[\cdot]$ and any $\Omega' \Subset \Omega$,
there exist constants $\rho_0<\textnormal{dist}(\Omega',\partial\Omega)$ and $C_2>0$ such that, for every
$z\in \Gamma_0\cap\Omega'$ and every $0<r<\rho_0$,
\[
\sup_{B_r(z)} |u| \le C_2 r^{1+\tau^\ast},
\]
where
\[\tau^\ast:= \min \left\{ \alpha_p - \epsilon, \frac{\gamma}{p-\gamma}\right\},\]
for any $0< \epsilon \ll 1$. The constants $\rho_0$ and $C_2$ depend only on $N,p,\gamma,\lambda_\pm$,
$\textnormal{dist}(\Omega',\partial\Omega)$, and the local $C^{1,\alpha}$ bounds from
\Cref{interior-reg,interior-reg-sing}.
\end{theorem}

\begin{proof}
We argue by contradiction. Suppose the conclusion fails. Then, for each \(k\in\mathbb N\),
there exists a local minimizer \(u_k\), a point
\[
z_k\in \Gamma_0(u_k)\cap \Omega',
\]
and a radius \(r_k\in (0,\rho_0/2)\) such that
\begin{equation}\label{Sk-def}
S_k:=\sup_{B_{r_k}(z_k)}|u_k|=k\,r_k^{1+\tau^\ast}.
\end{equation}
Without loss of generality, we may choose \(r_k\) maximal with such a property, and hence
\begin{equation}\label{maximal-rk}
\sup_{B_r(z_k)}|u_k|\le k\,r^{1+\tau^\ast}
\qquad \text{for every } \ \ \  r_k\le r<\rho_0/2.
\end{equation}
Define
\[
v_k(x):=\frac{u_k(r_kx+z_k)}{S_k},
\qquad x\in B_{\rho_0/(2r_k)}.
\]
Then, it is immediate that
\begin{equation}\label{vk-norm}
\sup_{B_1}|v_k|=1,
\end{equation}
and, since \(z_k\in \Gamma_0(u_k)\),
\begin{equation}\label{vk-zero}
v_k(0)=0,
\qquad
\nabla v_k(0)=0.
\end{equation}
Moreover, for every fixed \(R\ge1\), if \(k\) is sufficiently large so that
\(R<\rho_0/(2r_k)\), then by \eqref{maximal-rk},
\begin{equation}\label{vk-growth}
\sup_{B_R}|v_k|
=\frac{\sup_{B_{Rr_k}(z_k)}|u_k|}{S_k}
\le \frac{k(Rr_k)^{1+\tau^\ast}}{S_k}
= R^{1+\tau^\ast}.
\end{equation}
By the scaling invariance (cf. \Cref{scaling-rmk}), each \(v_k\) is a local minimizer of \(\J_{\delta_k}[\cdot]\) with
\[
\delta_k=r_k^p S_k^{\gamma-p},
\]
and we obtain, as $k \to \infty$,
\[
\delta_k
=r_k^p\bigl(k\,r_k^{1+\tau^\ast}\bigr)^{\gamma-p}
=k^{-(p-\gamma)}\,r_k^{\,p-(1+\tau^\ast)(p-\gamma)}
\leq k^{-(p-\gamma)} \,(\rho_0/2)^{\,p-(1+\tau^\ast)(p-\gamma)} \to 0,
\]
since $p-(1+\tau^\ast)(p-\gamma) \geq 0$. Additionally, we have 
\[\sup_{B_{2R}} |v_k| \leq 2^{1+\tau^\ast} R^{1+\tau^\ast}\] 
ensuring that $\{v_k\}_{k\in \mathbb{N}}$ is uniformly (with respect to $k$) bounded in $B_{2R}$, for any fixed $R\geq1$. Thus, by \Cref{interior-reg} and \Cref{interior-reg-sing}, the family
\(\{v_k\}_{k\in \mathbb{N}}\) is uniformly bounded in \(C^{1,\alpha}(B_R)\), for some \(\alpha\in(0,1)\)
independent of \(k\), and therefore, up to a subsequence,
\[
v_k\to v_\infty
\qquad \text{in } C^1_{\mathrm{loc}}(\mathbb R^N).
\]
Passing to the limit in \eqref{vk-norm}, \eqref{vk-zero}, and \eqref{vk-growth}, we get
\[
\sup_{B_1}|v_\infty|=1,
\qquad
v_\infty(0)=0,
\qquad
\nabla v_\infty(0)=0,
\]
and
\[
\sup_{B_R}|v_\infty|\le R^{1+\tau^\ast}
\qquad \text{for every } \ \ \ R\ge1.
\]
We claim that \(v_\infty\) is $p$-harmonic in \(\mathbb R^N\). Indeed, for every
\(\varphi\in C_0^\infty(B_R)\), the minimality of \(v_k\) gives
\[
\int_{B_R}\frac{|\nabla v_k|^p}{p}\,dx
+\delta_k\int_{B_R}F_\gamma(v_k)\,dx
\le
\int_{B_R}\frac{|\nabla(v_k+\varphi)|^p}{p}\,dx
+\delta_k\int_{B_R}F_\gamma(v_k+\varphi)\,dx.
\]
Since \(\delta_k\to0\), as well as \(v_k\to v_\infty\) and
\(\nabla v_k\to \nabla v_\infty\) locally uniformly, we may pass to the limit and infer that
\[
\int_{B_R}\frac{|\nabla v_\infty|^p}{p}\,dx
\le
\int_{B_R}\frac{|\nabla(v_\infty+\varphi)|^p}{p}\,dx,
\qquad \forall \varphi\in C_0^\infty(B_R).
\]
Thus, \(v_\infty\) is $p$-harmonic in \(B_R\). Since \(R\) is arbitrary,
\(v_\infty\) is $p$-harmonic in \(\mathbb R^N\). Now, we apply \Cref{Lv} using condition \(\tau^\ast<\alpha_p\), and conclude \(v_\infty\equiv 0\), which contradicts \(\sup_{B_1}|v_\infty|=1\). The theorem is proved.
\end{proof}

\subsubsection{Illustration of optimality}

The growth exponent granted by \Cref{Optimal-growth-thm} takes the form of a competition between a negative perturbation of the quantity $1+\alpha_p$ and $1+\tau$, with $\tau:= \gamma/(p-\gamma)$. The optimality can be illustrated provided $\tau \leq \alpha_p-\epsilon$, for $\epsilon \ll 1$, as follows. Consider the explicit function 
\begin{equation} \label{explicit-example}
    u_0(x):= C_+ (x_N)_+^{1+\tau} - C_- (x_N)_-^{1+\tau}, \qquad C_\pm:= \left(\dfrac{\lambda_\pm p(p-\gamma)^{p}}{p^{p}(p-1)} \right)^{\frac{1}{p-\gamma}},  
\end{equation}
where $x_N$ is the $N$th coordinate of $x$ and $C_\pm$ are obviously positive. We first demonstrate that $u_0$ is a local minimizer of $\J[\cdot]$. To this end, we define $\mathbf{F}(t):= \int_0^t \mathbf{F}'(s) ds$ with $t \in \R$ and
\begin{equation*}
    \mathbf{F}'(t):= \left( \dfrac{p}{p-1} F_\gamma(t)\right)^{\frac{p-1}{p}},
\end{equation*}
where $F_\gamma \geq 0$ is given by \eqref{Fg}.  Recalling \Cref{local-min}, let $\Omega' \Subset \Omega$ and $v \in W^{1,p}(\Omega')$ be such that $u_0-v \in W_0^{1,p}(\Omega')$. By construction,  
\begin{equation}\label{psi-asump}
    \dfrac{\big(\mathbf{F}'(t)\big)^{\frac{p}{p-1}}}{p/(p-1)} = F_\gamma(t).
\end{equation}
Moreover, the choices of $C_\pm$ ensure that
\begin{equation}\label{u0-asump}
    \big|\nabla u_0\big|^{p} = \big(\mathbf{F}'(u_0)\big)^{\frac{p}{p-1}}\geq0,
\end{equation}
where 
\begin{equation*} \label{u0-deriv}
    \nabla u_0 = \partial_{x_N} u_0 = \left[ C_+(1+\tau) (x_N)^{\tau} \chi_{\{x_N>0\}} + C_-(1+\tau) (-x_N)^{\tau} \chi_{\{x_N<0\}} \right]e_N.
\end{equation*}
Notice, furthermore, that $\mathbf{F}(w) \in W^{1,1}(\Omega')$ as long as $w \in W^{1,p}(\Omega')$. Hence, Young's inequality provides
\begin{equation} \label{v-ineq}
 \text{div} (\mathbf{F}(v) e_N) =   \mathbf{F}'(v) \partial_{x_N} v \leq \dfrac{|\nabla v|^p}{p} + F_\gamma(v),
\end{equation}
a.e. in $\Omega'$, where we have used \eqref{psi-asump}. On the other hand, \eqref{u0-asump} enforces a.e. in $\Omega'$ that 
\begin{equation} \label{u0-eq}
 \text{div} (\mathbf{F}(u_0) e_N) =   \mathbf{F}'(u_0) \partial_{x_N} u_0 = \dfrac{|\nabla u_0|^p}{p} + F_\gamma(u_0).
\end{equation}
 By the divergence theorem, $v$ satisfies
 \begin{equation*}
     \int_{\Omega'} \text{div} (\mathbf{F}(v) e_N) dx = \int_{\partial \Omega'} \mathbf{F}(v)e_N  \cdot \nu d \mathcal{H}^{N-1},  
 \end{equation*}
 and so does $u_0$. Since $u_0$ and $v$ agree on $\partial \Omega'$ in the sense of traces, so do $\mathbf{F}(u_0)$ and $\mathbf{F}(v)$; hence, one deduces 
\begin{equation} \label{interm-psi-ineq}
      \int_{\Omega'}  \mathbf{F}'(u_0) \partial_{x_N} u_0 dx = \int_{\Omega'}  \mathbf{F}'(v) \partial_{x_N} v dx. 
 \end{equation}
Combining \eqref{v-ineq}, \eqref{u0-eq}, and \eqref{interm-psi-ineq}, we conclude in $\Omega'$ that
\begin{equation*}
    \J[u_0] = \int_{\Omega'}\mathbf{F}'(u_0) \partial_{x_N} u_0 dx = \int_{\Omega'}\mathbf{F}'(v) \partial_{x_N} v dx \leq \J[v],
\end{equation*}
as desired.
\\
Next, if $z \in \partial\{(u_0)_\pm>0 \} = \{x_N =0 \}$,  then $z_N=0$ and $u_0(z) = \nabla u_0(z) = 0$. Therefore, $z \in \Gamma_0$, and if $x \in B_r(z)$, \eqref{explicit-example} satisfies
\begin{equation} \label{sup-eq}
    \sup_{B_r(z)} (u_0)_\pm = C_\pm r^{1+\tau},
\end{equation}
showing that exponent $1+\tau$ is optimal.
Indeed, if not, there would exist $\tilde{\tau} > \tau$ such that 
\[\sup_{B_r(z)} (u_0)_\pm \leq \tilde{C}_\pm r^{1+\tilde{\tau}},\]
for every sufficiently small $r$ and some $\tilde{C}_\pm>0$ independent of $r$. Together with \eqref{sup-eq}, it necessitates 
\[C_\pm \leq \tilde{C}_\pm r^{\tilde{\tau}-\tau} \to 0,\]
as $r\to 0$, a contradiction to $C_\pm>0$ in \eqref{explicit-example}.

\begin{remark} \label{restricted}
Observe that $1+\tau = 1+\tau^\ast$ provided 
\[\frac{\gamma}{p-\gamma} < \alpha_p \quad \Longleftrightarrow \quad \gamma <\frac{p\alpha_p}{1+\alpha_p}.\] 
For $p=2$, this amounts to restricting $\gamma \in (0,1)$.
\end{remark}

\subsection{Nondegeneracy} In this subsection, we establish the decay rate for local minimizers away from $\Gamma$. This estimate is crucial for establishing the local porosity of $\Gamma_0$ and for deriving density estimates therein.  

\begin{theorem} \label{nondegeneracy-strong} 
If $u$ is a local minimizer of $\J[\cdot]$ with $0 < \gamma \leq 1$, then there exist constants $c_\pm>0$, depending on $p,\gamma$, $\lambda_\pm$, and $N$ such that if $x_0 \in \overline{\{u_\pm>0\}}$, then 
\begin{equation*}
        \sup_{B_r(x_0)} u_\pm \geq c_\pm r^{1+\tau},
\end{equation*}
for any $r < \textnormal{dist}(x_0,\partial \Omega)$.  
\end{theorem}

\begin{proof}
The proof is identical to that of the one-phase problem in \cite[Proposition 4.1]{araujo2024geometric}. For completeness, we show the estimate for $u_-$. 
\\
By continuity, it suffices to consider $x_0 \in \{u<0\}$. Take $r>0$ such that $B_r(x_0) \Subset \Omega$. Set $w(x):= c_-|x-x_0|^{\frac{p}{p-1}}$, for some $c_->0$ to be chosen. Letting $\mu: = (p-\gamma)/(p-1)$, we compute formally in $B_r(x_0) \cap \{u<0 \}$,
\begin{align*}
    \Delta_p (u_-)^\mu &= \mu^{p-1} (u_-)^{-\gamma} \Big( (1-\gamma)|\nabla u_-|^p + u_- \Delta_p u_- \Big) \\
    &= \mu^{p-1} (u_-)^{-\gamma} \Big( (1-\gamma)|\nabla u_-|^p +\gamma \lambda_- (u_-)^\gamma \Big) \\
    &\geq \gamma \lambda_- \mu^{p-1},
\end{align*}
where we have employed \Cref{Dist-sol}, and 
\begin{equation*}
    \Delta_p w = c_-^{p-1} N \left( \dfrac{p}{p-1} \right)^{p-1}.
\end{equation*}
We may thus choose $c_-=c_-(N,p,\gamma,\lambda_-)>0$ small enough (see \cite{araujo2024geometric} for a detailed justification) such that 
\begin{equation*}
    \Delta_p (u_-)^\mu \geq \Delta_p w \quad \text{in} \quad  B_r(x_0) \cap \{u<0\}.
\end{equation*}
 As $(u_-(x_0))^\mu>0 =w(x_0)$, the comparison principle implies the existence of a point $y \in \partial ( B_r(x_0) \cap \{ u<0\})$ such that 
\begin{equation} \label{nondeg-1}
\big(u_-(y)\big)^\mu > w(y).
\end{equation}
Moreover, for $x \in B_r(x_0) \cap \partial \{u<0\}$, it holds that 
\[\big(u_-(x)\big)^\mu = 0 < w(x).\]
Therefore, $y \in \partial B_r(x_0) \cap \{u<0\}$ and  \eqref{nondeg-1} readily gives
\begin{equation*}
         c_-r^{1+\tau}= \big(w(y) \big)^\frac{1}{\mu} \leq \sup_{\partial B_r(x_0) \cap \{u<0\}} u_-,
\end{equation*}
from which the desired estimate follows. 
\end{proof}

\section{The free boundaries} \label{FB-sec}

In view of \Cref{restricted} and \Cref{nondegeneracy-strong}, we henceforth restrict the analysis to the range
\begin{equation} \label{gamma-range}
0 < \gamma < \min \left\{1, \dfrac{p \alpha_p}{1+\alpha_p} \right\}.
\end{equation}

\subsection{Porosity and Lebesgue measures} \label{poro-subsec}
Here, we first establish the local porosity of the set $\Gamma_0$ in $\Omega$ and subsequently the trivial local $N$-dimensional Lebesgue measure of $\Gamma$.

\begin{proposition} \label{porosity}
The set $\Gamma_0$ is locally porous in $\Omega$ with a porosity constant $\kappa \in (0,1)$ that depends on $p,\gamma,c_\pm,$ and $C_2$.
\end{proposition}

\begin{proof} 
Fix any $z \in \Gamma_0$  and $\Omega' \Subset 
\Omega$ such that $B_r(z) \subset \Omega'$ and $r < \rho_0$  ($\rho_0$ is fixed in \Cref{Optimal-growth-thm}). By \Cref{Optimal-growth-thm} and \Cref{nondegeneracy-strong}, there exists $y \in B_{r}(z)$ such that 
    \begin{equation*} \label{nondeg-bnd}
        \dfrac{c_0}{2}r^{1+\tau} \leq |u(y)| \leq C_2 d(y)^{1+\tau},
    \end{equation*}
    with $d(y):= \textnormal{dist}(y,\Gamma_0)$ and $c_0 := \min\{c_+,c_-\}$. Thus, one has
    \begin{equation*}
        \kappa r < d(y) \ \ \ \textnormal{if} \ \ \ \kappa:= \min\Big\{ 1/2,  \left( c_0/2 C_2 \right)^{\frac{1}{1+\tau}} \Big\}, 
    \end{equation*}
    so that $B_{\kappa r}(y) \subset B_{d(y)}(y) \subset B_r(z) \backslash \Gamma_0$, which proves the proposition. 
\end{proof}

\begin{corollary}\label{zero-lebesgue-0}
For any $\Omega' \Subset \Omega$, one has
$\mathcal{L}^N(\Gamma_0\cap \Omega')=0.$
\end{corollary}

\begin{proof}
Indeed, by \Cref{porosity}, the set 
$\Gamma_0\cap \Omega'$ is $\kappa$-porous, and hence
\[
\dim_{\mathcal H}(\Gamma_0\cap \Omega')
\le N-C\kappa^N < N,
\]
for some $C=C(N)>0$ (cf. \cite[Sec. 2]{koskela1997hausdorff}). Therefore,
$\Gamma_0\cap \Omega'$ has zero $N$-dimensional Lebesgue measure.
\end{proof}

\begin{corollary} \label{zero-lebesgue-non} For any $\Omega' \Subset \Omega$, $\mathcal{L}^N(\Gamma_* \cap \Omega')=0$.
\end{corollary}

\begin{proof}
For every $x\in\Gamma_*$, the interior $C^{1,\alpha}$ regularity of $u$ and the
implicit function theorem provide an open neighborhood $U_x$ in which
$\Gamma_*$ is contained in a $C^{1,\alpha}$ graph. Since $\Omega$ is second-countable, the cover $\{U_x\}_{x\in\Gamma_*}$ admits a countable subcover
$\{U_j\}_{j\in\mathbb N}$. Each $\Gamma_*\cap U_j$ has zero
$N$-dimensional Lebesgue measure. Hence,
\[
\mathcal{L}^N(\Gamma_*\cap\Omega')
\leq
\sum_{j=1}^{\infty}
\mathcal{L}^N(\Gamma_*\cap U_j)=0.
\]
\end{proof}

\subsection{Density estimates} \label{density-subsec}

With $\gamma$ in the range specified by \eqref{gamma-range}, we prove Lebesgue density estimates for neighborhoods of points in $\Gamma_0$. 
The following lemma provides lower-density estimates.

\begin{lemma} \label{lower-dens}
    If $z_\pm \in \Gamma_0 \cap  \partial\{u_\pm>0\}$, then there exists $c>0$ depending on $C_1$ and $c_\pm$ such that 
    
    \begin{equation*}
        \mathcal{L}^N (B_{c r}) \leq \mathcal{L}^N\big(\{0 < u_\pm \leq C_2 r^{1+\tau}\} \cap B_r(z_\pm)\big).  
    \end{equation*}
for any $r<\rho_0$, $\rho_0$ is the fixed by \Cref{Optimal-growth-thm}.
\end{lemma}

\begin{proof}
It suffices to show the estimate for $z_+ \in \Gamma_0 \cap \partial\{u_+>0\}$. Thanks to the optimal growth and nondegeneracy estimates for $u$, we have for all $r<\rho_0$
\begin{equation} \label{up-lo-bnd}
         c_+r^{1+\tau} \leq \sup_{B_r(z_+)} u_+ \leq C_2 r^{1+\tau}.
\end{equation}
Hence, there exists $y \in B_{r/2}(z_+) \cap \{ u_+>0 \}$ such that 
\begin{equation} \label{uyr}
          u_+(y) \geq \dfrac{c_+}{2}r^{1+\tau}.
\end{equation}
For each $r\in(0,\rho_0)$, consider the rescaling
\begin{equation} \label{ur}
         u_r(x):= \dfrac{u(rx+y)}{r^{1+\tau}}, \quad x \in B_{1},
\end{equation}
which minimizes $\J[\cdot]$ locally in $B_1$ (cf. Section \ref{scaling-rmk}) and satisfies $\sup_{B_{1}} |u_r| \leq C_2$ by virtue of \eqref{up-lo-bnd}. In particular,  
\begin{equation*}
    u_r(x) \leq C_2,
\end{equation*}
all $x$ satisfying $|x|< 1$. By  \Cref{interior-reg} and \Cref{interior-reg-sing}, $u_r$ is $L$-Lipschitz continuous in $B_{1/2}$ for some $0< L \leq C_1$, i.e., $|u_r(x) - u_r(0)| \leq L |x|$, for $x \in B_{1/2}$. Now fix 
\begin{equation*}
    c := \min\left\{1/4, c_+/4L \right\}.
\end{equation*}
Obviously, $B_{c r}(y) \subset B_r(z_+)$ since
\begin{equation*}
    |x-z_+| \leq |x-y| + |y -z_+| \leq c r + \dfrac{r}{2} < r.
\end{equation*}
In view of \eqref{uyr} and \eqref{ur}, the blowup at $y$ satisfies
\begin{equation*}
         u_r(0) \geq \dfrac{c_+}{2},
\end{equation*} 
and, thus, for all $x$ with $|x| \leq c$,
\begin{equation*}
         u_r(x) \geq u_r(0) - L|x|\geq \dfrac{c_+}{2} - c L \geq \dfrac{c_+}{4}.
\end{equation*}
It follows that 
\[0<\dfrac{c_+}{4} \leq u_r \leq C_2  \quad \text{in } B_c,\] 
and, upon unscaling, we deduce $B_{c r}(y) \subset  \{0 < u_+ \leq C_2r^{1+\tau}\}$ from which the desired estimate follows.
\end{proof} 

We also have the following upper density estimate.

\begin{lemma} \label{upper-dens}
    If $z_\pm \in \Gamma_0 \cap \partial \{u_\pm >0\}$ and $\Omega' \Subset \Omega$ such that $B_r(z_\pm) \subset \Omega'$,  then
    \begin{equation*}
          \mathcal{L}^N \big( \{ 0<u_\pm \leq C_2 r^{1+\tau} \} \cap \Omega' \big) \leq Cr^{(1+\tau)(1-\gamma)},
    \end{equation*}
    where $C>0$ is independent of $r$.
\end{lemma}
\begin{proof}
    By \Cref{Loc-integ}, we deduce, for any $\gamma$ satisfying \eqref{gamma-range} and $\ell>0$,
    \begin{equation*}
        \int_{\{0<u_\pm \leq \ell\} \cap \Omega'} \ell^{\gamma -1} dx \leq \dfrac{1}{\gamma \lambda_\pm}\int_{\{0<u_\pm \leq \ell\} \cap \Omega'} \gamma \lambda_\pm u_\pm^{\gamma -1} dx \leq C,
    \end{equation*}
    with $C>0$ independent of $\ell$. In particular, setting $\ell:= C_2 r^{1+\tau}$, we obtain
    \begin{equation*}
        C^{\gamma-1}_2 r^{(1+\tau)(\gamma-1)}\mathcal{L}^N\big(\{0<u_\pm \leq C_2 r^{1+\tau}\} \cap \Omega' \big) \leq C.
    \end{equation*}
\end{proof}

\subsection{Hausdorff measure estimates} \label{hausdorff-sec} We conclude in this subsection, again under the restriction \eqref{gamma-range}, with Hausdorff measure estimates for the free boundaries. We first deal with the nonbranching part in the following proposition.

\begin{proposition}\label{hausdorff-nonbranch}
Let $u$ be a local minimizer of $\J[\cdot]$ and $\Omega'\Subset\Omega$.
For every $\eta>0$,
\[
\mathcal{H}^{N-1}\left(
\Gamma_*\cap\overline{\Omega'}\cap\{|\nabla u|\geq\eta\}
\right)<\infty.
\]
Consequently, $\Gamma_*\cap\Omega'$ is countably $C^{1,\alpha}$ rectifiable
and $\mathcal{H}^{N-1}$-$\sigma$-finite. In particular,
\[
\mathcal{H}^{s}(\Gamma_*\cap\Omega')=0
\qquad\text{for every } \ \ s>N-1,
\]
and
\[
\textnormal{dim}_\mathcal{H}(\Gamma_*) \leq N-1.
\]
\end{proposition}

\begin{proof}
For $\eta>0$, set
\[
E_\eta:=
\Gamma_*\cap\overline{\Omega'}\cap\{|\nabla u|\geq\eta\}.
\]
Equivalently,
\[
E_\eta=
\bigl(\partial\{u>0\}\cap\partial\{u<0\}\bigr)
\cap\overline{\Omega'}\cap\{|\nabla u|\geq\eta\}.
\]
The latter representation shows that $E_\eta$ is compact. At every point of
$E_\eta$, the implicit function theorem represents the zero level set of $u$,
and hence of $\Gamma_*$, locally in $\Omega$ as a $C^{1,\alpha}$ graph. Compactness, therefore,
provides a finite covering of $E_\eta$ by such graph neighborhoods. The area
formula yields
\[
\mathcal{H}^{N-1}(E_\eta)<\infty.
\]
Finally,
\[
\Gamma_*\cap\Omega'
\subseteq
\bigcup_{m=1}^{\infty}E_{1/m}.
\]
Thus $\Gamma_*\cap\Omega'$ is countably $C^{1,\alpha}$ rectifiable and
$\mathcal{H}^{N-1}$-$\sigma$-finite. The final assertion follows because a set
of finite $\mathcal{H}^{N-1}$ measure has zero $\mathcal{H}^{s}$ measure for
every $s>N-1$.
\end{proof}
 
The following theorem establishes the local finiteness of the Hausdorff measure of $\Gamma_0$ at a particular dimension dictated by the estimates developed earlier.     

\begin{theorem}
If $u$ is a local minimizer of $\J[\cdot]$ and $\Omega' \Subset \Omega$, then
\[
\mathcal{H}^{N-1+\tau(p-1)}\big(\Gamma_0 \cap \Omega' \big)<\infty. 
\]
\end{theorem}
\begin{proof}
    We present the argument for 
    \begin{equation*}
        \Gamma^+_0 := \Gamma_0 \cap \partial\{  u_+ >0 \}.
    \end{equation*}
    Fix a subset $\Omega' \Subset \Omega$ and a radius $\rho_*$ such that
    \begin{equation*}
        0 < \rho_* < \dfrac{\rho_0} {4},
    \end{equation*}
where $\rho_0 < \textnormal{dist}(\Omega',\partial  \Omega)$ is prescribed by \Cref{Optimal-growth-thm}.
For $0 < r < \rho_*/4$, the family
\[
\mathcal F:=\{\overline{B}_{r}(z): z\in \Gamma^+_0\cap\Omega'\}
\]
is a Besicovitch-type cover of $\Gamma^+_0\cap\Omega'$. By Besicovitch's covering theorem, there exists an $m=m(N)$ subcollections $\{\mathcal{G}_k\}_{k=1}^m$, each consisting of pairwise disjoint closed balls contained in $\mathcal{F}$, such that for $z_i \in \Gamma^+_0\cap\Omega'$,
    \[
\Gamma^+_0\cap\Omega' \subset \bigcup_{k=1}^m\bigcup_{ \overline{B}_{r}(z_i) \in \mathcal{G}_k} \overline{B}_{r}(z_i).
\]
Define the set
\begin{equation*}
    \mathcal{N}_{\rho_*}(\Gamma^+_0\cap\Omega'):= \big\{x: \textnormal{dist}(x,\Gamma^+_0\cap\Omega')< \rho_*\} \Subset \Omega,
\end{equation*}
and notice that $\overline{B}_r(z_i) \Subset \mathcal{N}_{\rho_*}(\Gamma^+_0 \cap\Omega')$. As illustrated in the proof of \Cref{lower-dens}, for each $z_i$, there exist $c_i>0$ independent of $r$ and $y_i \in B_{r/2}(z_i) \cap \{u_+>0\}$ such that $B_{c_i r}(y_i) \subset B_r(z_i)$, $B_{c_i r}(y_i)$ are  evidently pairwise disjoint in $\mathcal{G}_k$. It follows by \Cref{lower-dens}, and \Cref{upper-dens} that
\begin{align*}
      \textnormal{Card } \mathcal{G}_k \cdot C(N)  c^N r^N &\leq \sum_{i=1}^{\textnormal{Card } \mathcal{G}_k} \mathcal{L}^N(B_{c_i r}(y_i))  \nn \\
    &= \mathcal{L}^N \left( \bigcup_{i=1}^{\textnormal{Card } \mathcal{G}_k} B_{c_i r}(y_i)\right) \nn \\
    & \leq \mathcal{L}^N\left( \big\{0 < u_+\leq C_2 r^{1+\tau} \big\} \cap \mathcal{N}_{\rho_*}\big(\Gamma^+_0\cap\Omega'\big)\right)\nn \\
    &\leq Cr^{(1+\tau)(1-\gamma)} 
\end{align*}
    with $C>0$ and $c := \min_i c_i$ being independent of $r$. 
    Therefore, for each $k \in \{1,\dots,m\}$
\begin{equation*}
    \textnormal{Card } \mathcal{G}_k \leq Cr^{(1+\tau)(1-\gamma)-N}, 
\end{equation*}
and so we have constructed the finite cover
    \[
\Gamma^+_0\cap\Omega' \subset \bigcup_{j=1}^M \overline{B}_{r}(z_j),
\]
with $M := \sum_{k=1}^m \textnormal{Card } \mathcal{G}_k \leq mCr^{(1+\tau)(1-\gamma)-N} $. Accordingly, we estimate
\begin{equation*}
    \mathcal H^{N-1+\tau(p-1)}_{4r}(\Gamma^+_0\cap\Omega') \leq \sum_{j=1}^M (\text{diam} \ \overline{B}_{r} (z_j) )^{N-1+\tau(p-1)} \leq CMr^{N-1+\tau(p-1)} \leq C,
\end{equation*}
    and send $r \to 0$ to complete the proof.
\end{proof}

The following straightforward corollary sharpens the Hausdorff dimension drop of $\Gamma_0$ consequent upon its local porosity.  
 
\begin{corollary} The Hausdorff dimension of the set $\Gamma_0$ is less than or equal to $N-1+\tau(p-1)$.
\end{corollary}
 
\appendix

\section{Some technical lemmas} \label{Apndx}

\begin{lemma}[{{\cite[Lemma 1.1] {maly1997fine}}}] \label{sum-p-harm} For any $a,b \in \R^N$ and $\varepsilon>0$, we have
\begin{equation*}
    |a+b|^p \leq \begin{cases}
        |a|^p + |b|^p \quad \textnormal{if} \quad 0 < p \leq 1, \\
        (1+\varepsilon)^{p-1} |a|^p + (1+\varepsilon^{-1})^{p-1} |b|^p \quad \textnormal{if} \quad 1 \leq p < \infty.
    \end{cases}
\end{equation*}
\end{lemma}

The following iteration lemmas are used in the proofs of local boundedness and regularity.

\begin{lemma} [{{\cite[Lemma 3.1 of Ch. V]{giaquinta1983multiple}}}] \label{hole-filling}
    Let $\phi(t)$ be a nonnegative and bounded function in $[r_0,R_0]$ with $r_0\geq0$. Suppose that for $r_0 \leq t < s \leq R_0$
    \begin{equation*}
        \phi(t) \leq   \theta \phi(s) +  \left[ \dfrac{A}{(s-t)^\alpha}+ B \right],
    \end{equation*}
    where $A,B,\alpha,$ and $\theta$ are nonnegative constants and $\theta < 1$. Then, for all $r_0 \leq r < R \leq R_0$, it holds that 
    \[\phi(r) \leq   C(\alpha,\theta)\left[ \dfrac{A}{(R-r)^\alpha}+ B \right].\]
\end{lemma}

\begin{lemma} [{{\cite[Lemma 2.1 of Ch. III]{giaquinta1983multiple}}}; see also {{\cite[Remark 7.8]{giusti2003direct}}}] \label{Algebraic-ineq}
    Let $\phi$ be a nonnegative and nondecreasing function. Suppose that 
    \begin{equation*}
        \phi(r) \leq   A \left[\left(\dfrac{r}{R} \right)^\alpha + \mu \right]\phi(R) + B R^{\beta},
    \end{equation*}
    for $r \leq R \leq R_0$, with nonnegative constants $A,\alpha,\beta$ such that $ \beta < \alpha$. Then, there exists a constant $\mu_0 = \mu(A,\alpha,\beta)$ such that if $\mu < \mu_0$, then 
    \begin{equation} \label{iter-conc}
        \phi(r) \leq C \left[ \left( \dfrac{1}{R}\right)^\beta \phi(R) + B  \right] r^\beta,
    \end{equation}
    for all $r \leq R \leq R_0$ where $C$ depends on $A,\alpha,$ and $\beta$. Alternatively, if $\beta \geq \alpha$, then \eqref{iter-conc} is replaced by
    \begin{equation} \label{iter-conc2}
        \phi(r) \leq C \left[ \left( \dfrac{1}{R}\right)^{\alpha-\epsilon} \phi(R) + B  \right]r^{\alpha-\epsilon},
    \end{equation}
    for any $0 < \epsilon < \alpha$. 
\end{lemma}

The following Liouville-type lemma is used in the proof of the optimal growth of minimizers. We include the proof for the reader’s convenience.

\begin{lemma}\label{Lv}
Let $u$ be a $p$-harmonic function in $\mathbb R^N$ such that
\[
u(0)=0
\qquad\text{and}\qquad
\nabla u(0)=0.
\]
Assume that
\begin{equation}\label{lv-bnd2}
\|u\|_{L^\infty(B_R)}\le C R^{1+\tau}
\qquad\text{for every } \ \ \  R\ge 1,
\end{equation}
for some $\tau\in (0,\alpha_p)$, where $\alpha_p\in(0,1)$ is the optimal H\"older exponent for the gradient of $p$-harmonic functions. Then $u\equiv 0$ in $\mathbb R^N$.
\end{lemma}

\begin{proof}
For $R\ge 1$, define the rescaled function
\[
u_R(x):=\frac{u(Rx)}{R^{1+\tau}}, \qquad x\in B_1.
\]
Since $u$ is $p$-harmonic in $\mathbb R^N$, each $u_R$ is $p$-harmonic in $B_1$.
Moreover, by \eqref{lv-bnd2},
\[
\|u_R\|_{L^\infty(B_1)} \le C,
\]
with $C$ independent of $R$. Also,
\[
u_R(0)=0
\qquad\text{and}\qquad
\nabla u_R(0)=R^{-\tau}\nabla u(0)=0.
\]
By the interior $C^{1,\alpha_p}$ estimate for $p$-harmonic functions, there exists
a constant $C_p=C_p(N,p)>0$ such that
\[
|\nabla u_R(x)-\nabla u_R(0)|
\le C_p |x|^{\alpha_p}
\qquad \text{for all } \ \ \  x\in B_{1/2}.
\]
Since $\nabla u_R(0)=0$, this yields
\[
|\nabla u_R(x)|\le C_p |x|^{\alpha_p}
\qquad \text{for all } \ \ \ x\in B_{1/2}.
\]
Fix $y\in \mathbb R^N$. Choose $R>2|y|$, so that $x:=y/R\in B_{1/2}$. Then
\[
\nabla u_R(x)=R^{-\tau}\nabla u(y),
\]
and it follows that
\[
R^{-\tau}|\nabla u(y)|
= |\nabla u_R(y/R)|
\le C_p \left(\frac{|y|}{R}\right)^{\alpha_p}.
\]
Therefore,
\[
|\nabla u(y)|
\le C_p |y|^{\alpha_p} R^{\tau-\alpha_p}.
\]
Recalling that $\tau<\alpha_p$ and letting $R\to\infty$ yield
\[
\nabla u(y)=0.
\]
As $y\in \mathbb R^N$ was arbitrary, we conclude that $\nabla u\equiv 0$ in $\mathbb R^N$.
Hence, $u$ is constant. Since $u(0)=0$, it follows that $u\equiv 0$.
\end{proof}

{\small \noindent{\bf Acknowledgments.} The authors thank the anonymous referee for the constructive feedback provided during the review process. This publication is based upon work supported by King Abdullah University of Science and Technology (KAUST) under Award No. ORFS-CRG12-2024-6430.}

\end{document}